\documentclass[12pt]{amsart}  

\usepackage[latin1]{inputenc}
\usepackage{amsmath, amsfonts, amssymb}
\usepackage{mathtools}
\usepackage{comment}
\usepackage{graphicx}
\usepackage{subfigure}
\usepackage{color}
\usepackage{hyperref}
\usepackage{verbatim}
\usepackage[all]{xy}
\usepackage{graphics}
\usepackage{pdfsync}
\usepackage{ytableau}
\usepackage{subfiles}
\usepackage[normalem]{ulem} 

\theoremstyle{plain}
\newtheorem{theorem}{Theorem}[section]

\newtheorem{lemma}[theorem]{Lemma}
\newtheorem{corollary}[theorem]{Corollary}

\newtheorem{algorithm}[theorem]{Algorithm}

\theoremstyle{remark}
\newtheorem{remark}[theorem]{Remark}

\numberwithin{equation}{section}

\newcommand{\bN}{\mathbb{N}}

\newcommand{\hn}{H_n(0)}
\newcommand{\smodule}{\mathbf{S}}
\newcommand{\sgrp}[1]{\mathfrak{S}_{#1}}

\newcommand{\des}{\mathrm{Des}} 
\newcommand{\comp}{\mathrm{comp}} 
\newcommand{\stan}{\mathrm{stan}}

\newcommand{\rtau}{{\tau}} 

\newcommand{\suchthat}{\;|\;}

\newcommand{\RT}{\ensuremath{\operatorname{RT}}}
\newcommand{\SRT}{\ensuremath{\operatorname{SRT}}}

\newcommand{\RCT}{\ensuremath{\operatorname{CT}}}
\newcommand{\SRCT}{\ensuremath{\operatorname{SCT}}}

\newlength\cellsize 
\savebox2{%
\begin{picture}(15,15)
\put(0,0){\line(1,0){15}}
\put(0,0){\line(0,1){15}}
\put(15,0){\line(0,1){15}}
\put(0,15){\line(1,0){15}}
\end{picture}}
\newcommand\cellify[1]{\def\thearg{#1}\def\nothing{}%
\ifx\thearg\nothing
\vrule width0pt height\cellsize depth0pt\else
\hbox to 0pt{\usebox2\hss}\fi%
\vbox to 15\unitlength{
\vss
\hbox to 15\unitlength{\hss$#1$\hss}
\vss}}
\newcommand\tableau[1]{\vtop{\let\\=\cr
\setlength\baselineskip{-16000pt}
\setlength\lineskiplimit{16000pt}
\setlength\lineskip{0pt}
\halign{&\cellify{##}\cr#1\crcr}}}
\savebox3{%
\begin{picture}(15,15)
\put(0,0){\line(1,0){15}}
\put(0,0){\line(0,1){15}}
\put(15,0){\line(0,1){15}}
\put(0,15){\line(1,0){15}}
\end{picture}}
\newcommand\expath[1]{%
\hbox to 0pt{\usebox3\hss}%
\vbox to 15\unitlength{
\vss
\hbox to 15\unitlength{\hss$#1$\hss}
\vss}}
\newcommand\bas[1]{\omit \vbox to \cellsize{ \vss \hbox to \cellsize{\hss$#1$\hss} \vss}}

\newcommand{\al}{\alpha}

\usepackage{tikz}
\usepackage{xcolor}
\usetikzlibrary{positioning,arrows,calc,intersections,shapes,decorations}
\newcommand{\PRCT}{\mathrm{PCT}}
\newcommand{\SPRCT}{\mathrm{SPCT}}
\newcommand{\st}{\ensuremath{\operatorname{st}}}
\newcommand{\equiva}{\sim _\alpha}
\newcommand{\Cat}{\sf Cat} 

\newcommand{\dyck}{\mathcal{D}} 
\newcommand{\ldyck}{\mathcal{LD}} 
\newcommand{\tree}{\mathcal{T}} 
\newcommand{\ltree}{\mathcal{T}^{\ell}} 
\newcommand{\prcttoldyck}{\texttt{SPCTtoLDyck}} 
\newcommand{\ldycktoprct}{\texttt{LDycktoSPCT}} 
\newcommand{\ldycktoltree}{\texttt{LDycktoLTree}} 
\newcommand{\ltreetoldyck}{\texttt{LTreetoLDyck}} 
\newcommand{\rasc}{\mathrm{rasc}}
\newcommand{\lasce}{\mathrm{lasc}}
\newcommand{\rdes}{\mathrm{rdes}}
\newcommand{\ldes}{\mathrm{ldes}}
\newcommand{\reversetopermuted}{\overline{\phi}}
\newcommand{\permutedtoreverse}{\overline{\rho}} 
\definecolor{ballblue}{rgb}{0.13, 0.67, 0.8}
\usepackage[normalem]{ulem} 
\newcommand{\DyckWord}{\mathrm{DyckWord}}
\newcommand{\NumSteps}{\mathrm{NumSteps}}
\newcommand{\upstep}{\sf u} 
\newcommand{\downstep}{\sf d} 
\newcommand{\dword}{\texttt{dw}} 
\newcommand{\ldword}{\texttt{ldw}} 
\newcommand{\mlpd}{\texttt{mlpd}} 
\newcommand{\lpath}{\text{left path }} 
\newcommand{\lpaths}{\text{left paths }} 
\newcommand{\rootof}[1]{\mathrm{root}(#1)} 
\newcommand{\udword}{\text{ud-word}} 
\newcommand{\Inv}{\mathrm{Inv}} 
\newcommand{\parent}{\mathrm{parent}} 
\usepackage{algorithm,algorithmicx, tabularx}
\usepackage[noend]{algpseudocode}
\makeatletter
\newcommand{\multiline}[1]{%
  \begin{tabularx}{\dimexpr\linewidth-\ALG@thistlm}[t]{@{}X@{}}
    #1
  \end{tabularx}
}
\makeatother

\algrenewcommand\algorithmicrequire{\textbf{Input:}}
\algrenewcommand\algorithmicensure{\textbf{Output:}}

\begin{document}

\title[Permuted composition tableaux]{Permuted composition tableaux, 0-Hecke algebra and labeled binary trees}

\author{V. Tewari}
\address{Department of Mathematics, University of Washington, Seattle, WA 98195, USA}
\email{\href{mailto:vasut@math.washington.edu}{vasut@math.washington.edu}}

\author{S. van Willigenburg}
\address{Department of Mathematics, University of British Columbia, Vancouver, BC V6T 1Z2, Canada}
\email{\href{mailto:steph@math.ubc.ca}{steph@math.ubc.ca}}

\thanks{
The second author was supported in part by the National Sciences and Engineering Research Council of Canada.}


\begin{abstract}
We introduce a generalization of semistandard composition tableaux called permuted composition tableaux.
These tableaux are intimately related to permuted basement semistandard augmented fillings studied by Haglund, Mason and Remmel.
Our primary motivation for studying permuted composition tableaux is to enumerate all possible ordered pairs of permutations $(\sigma_1,\sigma_2)$ that can be obtained by standardizing the entries in two adjacent columns of an arbitrary composition tableau.
We refer to such pairs as compatible pairs.
To study compatible pairs in depth, we define a $0$-Hecke action on permuted composition tableaux.
This action naturally defines an equivalence relation on these tableaux.
Certain distinguished representatives of the resulting equivalence classes in the special case of two-columned tableaux are in bijection with compatible pairs.
We provide a bijection between two-columned tableaux and labeled binary trees.
This bijection maps a quadruple of descent statistics for 2-columned tableaux to left and right ascent-descent statistics on labeled binary trees introduced by Gessel, and we use it to prove that the number of compatible pairs is $(n+1)^{n-1}$.
\end{abstract}

\maketitle
\tableofcontents

\section{Introduction}\label{sec:intro}
Composition tableaux were introduced in \cite{HLMvW-1} to define the basis of quasisymmetric Schur functions for the Hopf algebra of quasisymmetric functions.
These functions are analogues of the ubiquitous Schur functions \cite{schur}, have been studied in substantial detail recently \cite{BTvW, HLMvW-1, HLMvW-2, Lauve-Mason, LMvW, TvW}, {and have consequently been the genesis of an  active new branch of algebraic combinatorics discovering Schur-like bases in quasisymmetric functions \cite{AllenHallamMason, BBSSZ-0}, type $B$ quasisymmetric Schur functions \cite{JingLi, Oguz}, quasi-key polynomials \cite{AssafS, Searles} and quasisymmetric Grothendieck polynomials \cite{Monical}}.
Just as Young tableaux play a crucial role in the combinatorics of Schur functions \cite{sagan, stanley-ec2}, composition tableaux are key to understanding the combinatorics of quasisymmetric Schur functions.

The aim of this article is to shed  light on certain enumerative aspects of composition tableaux by studying actions of the $0$-Hecke algebra of type $A$.
Recall that the number of standard Young tableaux is given by the famous hook-length formula of Frame-Robinson-Thrall \cite{FRT}.
While there is no known `hook-length formula' for enumerating standard composition tableaux, and the presence of large prime factors in the data suggests there might not be a simple closed formula, we hope that our article convinces the reader of other interesting enumerative aspects of composition tableaux and motivates further exploration.

The following question serves as the primary motivation for this article: {Can we} characterize and enumerate the possible ordered pairs of permutations $(\sigma_1,\sigma_2)$ obtained by considering the relative order of entries in two adjacent columns of an arbitrary composition tableau?
This question in the context of Young tableaux is uninteresting.
On the other hand, it is  natural with regards to composition tableaux as the relative order of entries in{, say,} column $i$ of a composition tableau $\tau$ is governed by the relative order of the entries in column $i-1$.
This is because of the so-called triple condition/rule \cite[Definition 4.1]{HLMvW-1}.
If the columns under consideration have the same number of entries, say $n$,  then we say that the resulting pair $(\sigma_1,\sigma_2)$ is a \emph{compatible pair} of permutations in $\sgrp{n}$.
We remark here that the assumption that the two columns possess the same number of entries is  mild  as this can always be achieved by supplementing the composition tableau under consideration with $0$s as suggested in \cite[Definition 4.1 part 3)]{HLMvW-1}.
One of our main results is that the number of compatible pairs of permutations in $\sgrp{n}$ is $(n+1)^{n-1}$.
To study the combinatorics of an arbitrary pair of adjacent columns in a composition tableau, we introduce a generalization thereof that we call permuted composition tableaux.
Just as semistandard composition tableaux are intimately tied to semistandard augmented fillings \cite{HLMvW-1}, our permuted composition tableaux are connected to permuted basement semistandard  augmented fillings introduced by Haglund-Mason-Remmel \cite{HMR}.
Given our interest in the relative order of entries in columns,
the authors' previous work \cite{TvW} hints at constructing an action of the $0$-Hecke algebra $H_{n}(0)$ on permuted composition tableaux.
To this end, we generalize the operators from \cite{TvW} to our current setting and also obtain analogues of the results therein.

To answer our original question, we focus on standard permuted composition tableaux of shape $(2^n)$.
We refer to such tableaux as $2$-columned tableaux.
The $0$-Hecke action that we construct naturally leads to a deeper study of descents in the first column in $2$-columned tableaux and they naturally come in four flavors.
We construct a bijection between the set of $2$-columned tableaux  and labeled binary trees  that maps certain  descent statistics on tableaux to  ascent-descent statistics on labeled binary trees.
Note that the ascent-descent statistics on labeled binary trees were first studied in unpublished work by Gessel, and  functional equations for their
distribution were established by Kalikow \cite{Kalikow} and Drake
\cite{Drake-thesis}.
Given the striking observations of Gessel \cite{Gessel-Oberwolfach} relating the distribution of these
statistics to enumerative questions in the theory of hyperplane
arrangements, much work has been done recently
\cite{Bernardi, Corteel-Forge-Ventos,  Forge, GesselGriffinTewari, Tewari}.
That said, we will not focus on the hyperplane arrangements perspective in this article.

\smallskip
\noindent {\bf Outline of the article.} The paper essentially has two halves: the first half focuses on studying   the combinatorics of permuted composition tableaux and defining a $0$-Hecke action on them. The second half concerns itself  with enumerating and characterizing compatible pairs by way of Dyck paths and binary trees.

In Section \ref{sec:prelims}, we introduce our main combinatorial objects, and develop most of the notation we need.
Our main result here is Theorem~\ref{thm:generalized shift map} that makes explicit the link between reverse tableaux and permuted composition tableaux.
In Section~\ref{sec:0-Hecke}, we construct a $0$-Hecke action on the set of permuted composition tableaux of a given shape.
Section~\ref{sec:ascents-descents-trees} studies certain descents in $2$-columned tableaux and relates them to ascents-descents in labeled binary trees via
bijections that use labeled Dyck paths as intermediate objects.
Our main result in this section is Theorem~\ref{thm:stat-preserving bijection}.
Section~\ref{sec:allowable acyclic} gives a characterization for compatible pairs in terms of pattern-avoidance.
We demonstrate that compatible pairs are essentially allowable pairs introduced in \cite{AtkinsonThiyagarajah} and investigated further {in \cite{ALW, GilbeyKalikow, Hamel}}.
To answer our original question, we show in Corollary~\ref{cor:existence of srcts} that every allowable pair can be obtained by standardizing the entries in the last two columns of some standard composition tableau.

\section{Preliminaries}\label{sec:prelims}
We denote the set of positive integers by $\bN$. Given $n\in \bN$, we define $[n]$ to be the set of the first $n$ positive integers $\{1,\dots, n\}$.
The set of all words in the alphabet $\bN$ is denoted by $\bN^{*}$.
The empty word is denoted  by $\varepsilon$.
A \emph{\udword} is a word in the alphabet $\{\upstep,\downstep\}$.
A related notion is that of  a \emph{labeled ud-word}, which is a word in the alphabet $\{{\upstep}_i,{\downstep}_i\suchthat i\in \bN\}$.
There is a natural projection $\chi$ from the set of labeled ud-words to the set of ud-words defined by mapping ${\upstep}_i$ to $\upstep$ and ${\downstep}_{i}$ to $\downstep$ for all $i\in \bN$.

\subsection{The symmetric group $\sgrp{n}$}\label{subsec:symmetric group}
The symmetric group $\sgrp{n}$ is generated by the elements $s_1,\dots, s_{n-1}$ subject to the following  relations
\begin{eqnarray*}
	s_i^2&=&1 \text{ for } 1\leq i\leq n-1\\s_{i}s_{i+1}s_{i}&=&s_{i+1}s_is_{i+1} \text { for } 1\leq i\leq n-2\\s_{i}s_{j}&=&s_js_{i} \text{ if } \lvert i-j\rvert \ge 2.
\end{eqnarray*}
In the first relation above, the $1$ denotes the identity element in $\sgrp{n}$.
In practice, one identifies $\sgrp{n}$ with the group of all bijections from $[n]$ to itself, otherwise known as \emph{permutations}, by setting $s_i$ to be the \emph{simple transposition} interchanging $i$ and $i+1$ for $1\leq i\leq n-1$.
An expression for $\sigma\in \sgrp{n}$ of the form $s_{i_1}\cdots s_{i_p}$ that uses the minimal number of simple transpositions is called a \emph{reduced word} for $\sigma$.

We write permutations in one-line notation and, on occasion, treat them as words in $\bN^{*}$.
Given a word $w=w_1\cdots w_n\in \bN^{*}$, we define the \emph{standardization} of $w$, denoted by $\stan(w)$, to be the unique permutation $\sigma \in \sgrp{n}$ such that $\sigma(i) > \sigma (j)$ if and only if $w_i  > w_j$ for $1\leq i<j \leq n$.
For example, $\stan(3122)=4123$.
An \emph{inversion} of $\sigma\in\sgrp{n}$ is an ordered pair $(p,q)$  such that $1\leq p<q\leq n$ and $\sigma(p)>\sigma(q)$.
The set of inversions in $\sigma$ is denoted by $\Inv(\sigma)$.
This given, define the (left) \emph{weak Bruhat order} $\leq_{L}$ on $\sgrp{n}$ by $\sigma_1 \leq_{L} \sigma_2$ if and only if $\Inv(\sigma_1)\subseteq \Inv(\sigma_2)$ \cite[Proposition 3.1.3]{bjorner-brenti}.
We denote the cover relation in the weak Bruhat order by $\prec_L$.
The symmetric group $\sgrp{n}$ endowed with the weak Bruhat order $\leq_L$ inherits the structure of a graded lattice with a unique minimum element (given by the identity permutation) and a unique maximum element (given by the reverse of the identity permutation).
To emphasize the dependence on $n$, the identity permutation  in $\sgrp{n}$ and its reverse will henceforth be denoted by $\epsilon_n$ and $\bar{\epsilon}_n$ respectively.

\subsection{The \texorpdfstring{$0$-Hecke algebra $\hn$}{0-Hecke algebra}}\label{subsec:reps}
The $0$-Hecke algebra $H_n(0)$ is the $\mathbb{C}$-algebra generated by the elements $T_1,\ldots,T_{n-1}$ subject to the following relations
\begin{eqnarray*}
	T_i^2&=&T_i \text{ for } 1\leq i\leq n-1\\T_{i}T_{i+1}T_{i}&=&T_{i+1}T_iT_{i+1} \text { for } 1\leq i\leq n-2\\T_{i}T_{j}&=&T_jT_{i} \text{ if } \lvert i-j\rvert \ge 2.
\end{eqnarray*}

If $s_{i_1}\cdots s_{i_p}$ is a reduced word for a permutation $\sigma\in \sgrp{n}$,  then we define an element $T_{\sigma}\in H_{n}(0)$ as
\begin{eqnarray*}
	T_{\sigma}=T_{i_1}\cdots T_{i_p}.
\end{eqnarray*}
The Word Property \cite[Theorem 3.3.1]{bjorner-brenti} of $\sgrp{n}$  implies that $T_{\sigma}$ is independent of the choice of reduced word.
Moreover, the set $\{T_{\sigma}\suchthat \sigma\in \sgrp{n}\}$ is a linear basis for $H_{n}(0)$.
Thus, the dimension of $H_{n}(0)$ is $n!$.
In general, $H_{n}(0)$ is not semisimple and possesses a rich combinatorial representation theory \cite{carter, norton}. It is intimately related with the Hopf algebras of quasisymmetric functions and noncommutative symmetric functions respectively \cite{DKLT}, in the same way as symmetric group representation theory is intimately connected with the Hopf algebra of symmetric functions. More information about $H_{n}(0)$ and its representations can be found in \cite{carter,mathas}, and contemporary results can be found in \cite{BBSSZ, huang-1, huang-2, huang-3, Huang-Rhoades, Konig}.
Our interest in the $0$-Hecke algebra stems from the authors' previous work in the context of providing a representation-theoretic interpretation for quasisymmetric Schur functions \cite{TvW}.

\subsection{Compositions, partitions and diagrams}\label{subsec:compositions et cetera}
A \emph{composition} $\alpha$ of a positive integer $n$ is a finite ordered list of positive integers $(\alpha_1,\dots, \alpha_k)$ satisfying $\sum _{i=1}^k\alpha_i=n$.
The $\alpha_i$ are called the \emph{parts} of $\alpha$ and their sum is called the \emph{size} of $\alpha$ (denoted by $|\alpha|$).
The number of parts is called the \emph{length} of $\alpha$ and is denoted by $\ell(\alpha)${, and we often denote $m$ consecutive parts $j$ by $j^m$}.
The \emph{empty composition}, denoted by $\varnothing$, is the unique composition of size and length $0$.
If $\alpha$ is a composition of size $n$, we denote this by $\alpha\vDash n$.
From $\alpha=(\alpha_1,\dots,\alpha_k)\vDash n$, we can obtain another composition $\hat{\al}\vDash 2n$ of length $n$ by first incrementing each part of $\alpha$ by $1$ and subsequently adding $n-k$ parts equaling $1$ at the end.
For instance, if $\al=(2,1,3)$, then $\hat{\al}=(3,2,4,1,1,1)$.
If $\alpha=(\al_1,\dots,\al_k)\vDash n$ is such that $\al_1\geq \dots \geq \al_k$, then we say that $\al$ is a \emph{partition} of $n$ and denote this by $\al\vdash n$.
The partition obtained by sorting the parts of $\alpha$  in weakly decreasing order is denoted by $\widetilde{\alpha}$.

We depict a composition $\alpha=(\alpha_1,\dots,\alpha_k)\vDash n$ by its \emph{composition diagram}, also referred to as $\alpha$, which is a left-justified array of $n$ cells  where the $i$-th row from the top has $\alpha_i$ cells.
The \emph{augmented composition diagram} of $\alpha$, denoted by $\hat{\alpha}$, is the composition diagram of $\hat{\al}$.
We refer to the first column of $\hat{\al}$ as its \emph{basement}.
Note that if $\alpha$ is a partition, then the composition diagram of $\alpha$ is the Young diagram of $\alpha$ in English notation.

\vspace*{-1mm}
\subsection{Tableaux}\label{subsec:tableaux}
Given a partition $\lambda\vdash n$, a \emph{reverse tableau} (henceforth abbreviated to $\RT$) $T$ of \emph{shape} $\lambda$ is a filling of the Young diagram of $\lambda$ with positive integers such that the rows decrease weakly from left to right, whereas the columns decrease strictly from top to bottom.
Let $\RT(\lambda)$ denote the set of reverse tableaux of shape $\lambda$ all of whose entries are weakly less than  $|\lambda|$.
If $T\in \RT(\lambda)$ is such that its entries are all distinct, then we say that $T$ is \emph{standard}.
We denote the set of standard $\RT$s of shape $\lambda$ by $\SRT(\lambda)$.
Additionally, we  refer to a standard $\RT$ as an $\SRT$.

Given $\alpha\vDash n$ and $\sigma\in \sgrp{\ell(\al)}$, we define a \emph{permuted composition tableau} (henceforth abbreviated to $\PRCT$) of \emph{shape} $\al$ and \emph{type} $\sigma$ to be a filling $\tau$ of the composition diagram $\al$ with positive integers such that the following conditions hold.
\begin{enumerate}
	\item The entries in the first column are all distinct and  the standardization of the word obtained by reading the first column from top to bottom is $\sigma$.
	\item The entries along the rows decrease weakly when read from left to right.
	\item For any configuration in $\tau$ of the type in Figure~\ref{fig:triple configuration}, if $a\geq c$ then $b>c$. We call this condition the \emph{triple condition}.
\end{enumerate}

\begin{figure}[ht]
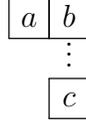

	\centering
	\begin{align*}
	\ytableausetup{mathmode,boxsize=1.25em}
	\begin{ytableau}
	a & b\\
	\none & \none[\vdots]\\
	\none & c
	\end{ytableau}
	\end{align*}
	\caption{A triple {configuration.}}
	\label{fig:triple configuration}
\end{figure}
The first condition and the triple condition guarantee that the entries in any given column of a $\PRCT$ are all distinct.
Figure~\ref{fig:prct} gives a $\PRCT$ of shape $(1,3,2,4)$ and type $1324$.

\begin{remark}
	Note that our version of the triple condition follows \cite[Definition 4.2.6]{LMvW}.
	The version in \cite[Definition 4.1 part 3)]{HLMvW-1}, though stated differently, is equivalent. It is the latter that we employ in Section~\ref{sec:allowable acyclic}.
\end{remark}
Let $\PRCT^{\sigma}(\alpha)$ denote the set of all $\PRCT$s  of shape $\alpha$ and type $\sigma $ all of whose  entries are weakly less than $|\al|$.
Additionally, let $$\PRCT(\alpha)\coloneqq\coprod_{\sigma\in \sgrp{\ell(\al)}}\PRCT^{\sigma}(\al),$$
where $\coprod$ indicates disjoint union.
If $\tau\in \PRCT^{\sigma}(\al)$ is such that its entries are all distinct, then we say that $\tau$ is \emph{standard}.
We denote the set of standard $\PRCT$s of shape $\alpha$ and type $\sigma$ by $\SPRCT^{\sigma}(\al)$.
Additionally, we  refer to a standard $\PRCT$ as an $\SPRCT$.
Finally, let $$\SPRCT(\alpha)\coloneqq\coprod_{\sigma\in \sgrp{\ell(\al)}}\SPRCT^{\sigma}(\al).$$

\begin{remark}
	From this point on, whenever we write $\PRCT^{\sigma}(\alpha)$ or $\SPRCT^{\sigma}(\alpha)$ without explicitly specifying $\sigma$, it is implicit that $\sigma$ is a permutation in $\sgrp{\ell(\alpha)}$. 
	Furthermore, all the RTs (respectively $\PRCT$s) we consider in this article belong to  $\RT(\lambda)$ (respectively $\PRCT(\alpha)$) for the appropriate $\lambda$ (respectively $\alpha$).
\end{remark}
\begin{figure}[ht]
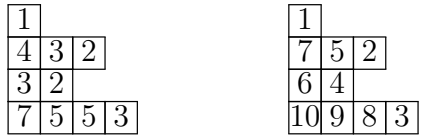

	\centering
	\begin{align*}
	\ytableausetup{mathmode,boxsize=1em}
  \begin{ytableau}
	1\\
	4 & 3 &2\\
	3 & 2\\
	7 & 5 & 5 & 3
\end{ytableau}\hspace{20mm}
	\begin{ytableau}
	1\\
	7 & 5 &2\\
	6 & 4\\
	10 & 9 & 8 & 3
	\end{ytableau}
	\end{align*}
	\caption{A $\PRCT$ (left) and an $\SPRCT$ (right) of shape $(1,3,2,4)$ and type $1324$.}
	\label{fig:prct}
\end{figure}

Given $\alpha\vDash n$ and $\tau\in \SPRCT(\al)$, we say that an integer $1\leq i\leq n-1$ is a \emph{descent} of $\tau$ if $i+1$ lies weakly right of $i$ in $\tau$.
The \emph{descent set} of $\tau$, denoted by $\des(\tau)$, consists of all descents in $\tau$.
For instance, the descent set of the $\SPRCT$ in Figure~\ref{fig:prct} is $\{1,2,4,6,7\}$.

\begin{remark}
	Note that if $\sigma$ equals $\epsilon_{\ell(\alpha)}$, then a $\PRCT$ of shape $\alpha$ and type $\sigma$ is in fact a \emph{composition tableau} (abbreviated henceforth to $\RCT$) introduced in \cite[Definition 4.1]{HLMvW-1} of the same shape $\alpha$.
  Similarly, an $\SPRCT$ of shape $\alpha$ and  type $\sigma=\epsilon_{\ell(\alpha)}$ corresponds to a \emph{standard $\RCT$} (abbreviated henceforth to $\SRCT$) of shape $\alpha$.
\end{remark}

Let $\alpha \vDash n$  whose largest part is $\alpha _{max}$ and let $\rtau \in \PRCT^{\sigma}(\alpha)$.
Suppose that the entries in column $i$ for $1\leq i \leq \alpha_{max}$ read from top to bottom form some word $w^i$ in $\bN^{*}$.
We refer to $w^i$ as the \emph{$i$-th column word} of $\tau$.
Furthermore, we define the \emph{standardized $i$-th column word} of $\rtau$, denoted by $\st_i(\rtau)$, to be $\stan(w^i)$.
Note that $\st_1(\tau)$ is $\sigma$ since $\tau\in \PRCT^{\sigma}(\al)$.
We define the \emph{standardized column word} of $\rtau$, denoted by $\st (\rtau)$, to be the word
$$\st (\rtau) = \st_1(\rtau) \  \st_2(\rtau) \cdots \st_{\alpha _{max}}(\rtau).$$
For the $\PRCT$ in Figure~\ref{fig:prct}, the standardized column word is $1324 \ 213 \ 12 \ 1$.

We now discuss two procedures connecting $\RT$s and $\PRCT$s.
The reader may interpret this correspondence as the analogue to that between $\RCT$s and semistandard augmented  fillings \cite{Mason-SLC}.
Consider a composition $\al$ and permutation $\sigma\in \sgrp{\ell(\al)}$.
Let $\al_{max}$ be the largest part of $\al$.
Our first procedure, $\permutedtoreverse_{\sigma}$, takes $\tau\in \PRCT^{\sigma}(\alpha)$ as input and outputs a filling $T$ of shape $\lambda\coloneqq\widetilde{\al}$ by considering the entries of column $i$ of $\tau$ in decreasing order and putting them in column $i$ of $\lambda$ from top to bottom for all $1\leq i\leq \al_{max}$.

For our second procedure $\reversetopermuted_{\sigma}$, let $\lambda$ be a partition and let $\sigma\in \sgrp{\ell(\lambda)}$.
This procedure takes $T\in \RT(\lambda)$ as input and outputs a filling $\tau$ as follows.
\begin{enumerate}
	\item Consider the entries in the first column of $T$ and write them in rows $1, 2, \ldots, \ell(\lambda)$ so that the standardization of the word obtained by reading from top to bottom is $\sigma$.
	\item Consider the entries in column $2$ in decreasing order and place each of them in the row with the smallest index so that the cell to the immediate left of the number being placed is filled and the row entries weakly decrease when read from left to right.
	\item Repeat the previous step with the set of entries in column $k$ for $k= 3, \ldots , \lambda_1$.
\end{enumerate}
Figure~\ref{fig:prct<->rt} illustrates the procedures just introduced and motivates the theorem that follows.
\begin{figure}[htbp]
	\centering
	\begin{align*}
	\ytableausetup{mathmode,boxsize=1em}
	\begin{ytableau}
	11 & 8 & 6 & 4\\
	10 & 7 & 5\\
	9 & 3 &1\\
	2
	\end{ytableau}
	\quad \mathrel{\mathop{\rightleftarrows}^{\mathrm{\reversetopermuted_{\sigma}}}_{\mathrm{\permutedtoreverse_{\sigma} }}} \quad
	\begin{ytableau}
	10 & 8 & 6 & 4 \\
	2\\
	11 & 7 & 5 \\
	9 & 3 & 1
	\end{ytableau}
	\end{align*}
	\caption{{An $\RT$} and a $\PRCT$ related via the maps $\permutedtoreverse_{\sigma}$ and $\reversetopermuted_{\sigma}$ where $\sigma=3142$.}
	\label{fig:prct<->rt}
\end{figure}

\begin{theorem}\label{thm:generalized shift map}
	Let $\lambda$ be a partition and let $\sigma\in \sgrp{\ell(\lambda)}$.
	The map
	$$\permutedtoreverse_{\sigma}:\coprod_{\widetilde{\alpha}=\lambda}\PRCT^{\sigma}(\alpha) \to \RT(\lambda)$$
	is a bijection and its inverse is the map $\reversetopermuted_{\sigma}$.
\end{theorem}
\begin{proof}
	First, we  show that $\permutedtoreverse_{\sigma}$ and $\reversetopermuted_{\sigma}$ are both injections, which suffices to conclude that they are in fact bijections.
	To this end, we will use \emph{permuted basement semistandard augmented fillings} (henceforth abbreviated to PBFs) introduced in \cite{HMR} generalizing the notion of semi-standard augmented fillings from \cite{Mason}.
	We refer the reader to \cite{HMR} for details on the terminology used in our proof.
	The reader will benefit from referring to Figure~\ref{fig:HMR} while reading this proof.

	Let $\lambda\vdash n$ and let $k=\ell(\lambda)$.
	Consider $\tau \in \PRCT^{\sigma}(\alpha)$ where $\widetilde{\alpha}=\lambda$.
	Suppose that the first column word of $\tau$ is $w=w_1\cdots w_{k}$.
	Clearly $\stan(w)=\sigma$.
	Let $\hat{\sigma}\in \sgrp{n}$ be the permutation  obtained by concatenating the entries in $[n]\setminus \{w_1,\dots, w_k\}$ at the end of $w$ in increasing order.
	Consider the filling $\hat{\tau}$ of the augmented composition diagram $\hat{\al}$ constructed as follows.
	\begin{itemize}
		\item The  basement contains $\hat{\sigma}(i)$ for $i=1$ through $n$ from top to bottom.
		\item The rest of the diagram, which is essentially the composition diagram of  $\al$, is filled exactly as $\tau$.
	\end{itemize}
	Given its construction, one may check that $\hat{\tau}$ is a PBF with basement permutation $\hat{\sigma}$.
	Indeed, the triple condition satisfied by $\tau$ ensures that, in $\hat{\tau}$, all type $A$ and $B$ triples are inversion triples and the $B$-increasing condition is satisfied.
	We omit the details.
	Crucial for us is the fact that the association of $\hat{\tau}$ to $\tau$ is  one-to-one. 
	 Next, we describe a map that allows us to associate a reverse tableau with $\hat{\tau}$.

	Given a permutation $\pi$, Haglund-Mason-Remmel \cite[Section 4]{HMR} define a map $\rho_{\pi}$ that generalizes Mason's shift map \cite{Mason} (also known as the $\rho$ map).  
	This map takes as input a PBF $T_1$ with basement permutation $\pi$ and outputs the unique PBF $T_2$ such that the entries in any column of $T_2$ read from top to bottom are the entries in the corresponding column of $T_1$ read in decreasing order.
  Note that the basement permutation of $T_2$ is $\bar{\epsilon}_n$.
  Haglund-Mason-Remmel (see discussion after \cite[Corollary 9]{HMR}) establish that $T_2$ is in fact a reverse tableau.
  The map $\rho_{\pi}^{-1}$ \cite[Page 309]{HMR} is the same as our map $\bar{\phi}_{\pi}$.
  In Figure~\ref{fig:HMR}, the two fillings in the middle are PBFs related by $\rho_{\pi}$ where $ \pi=10\ 2 \ 11\ 9\ 1 \ 3 \ 4\ 5 \ 6 \ 7 \ 8$ in one-line notation.

	In light of the preceding discussion, using the map $\rho_{\hat{\sigma}}$ we can associate the reverse tableau $\hat{T}\coloneqq \rho_{\hat{\sigma}}(\hat{\tau})$ with $\hat{\tau}$.
	Furthermore, given how $\rho_{\hat{\sigma}}$ operates, we conclude that the reverse tableau $T$ obtained by omitting the basement $\bar{\epsilon}_n$ from $\hat{T}$ may be otherwise obtained by sorting the entries in individual columns of $\tau$ in decreasing order and writing them along columns in the Young diagram of $\lambda=\widetilde{\alpha}$.
	Thus, we have $\permutedtoreverse_{\sigma}(\tau)=T$.

	The injectivity of $\permutedtoreverse_{\sigma}$ is explained next.
	Suppose that $\tau_1$ and $\tau_2$ satisfy $\permutedtoreverse_{\sigma}(\tau_1)= \permutedtoreverse_{\sigma}(\tau_2)$.
	Let $\hat{\sigma}_1, \hat{\sigma}_2\in \sgrp{n}$ be obtained by extending the first column words in $\tau_1$ and $\tau_2$ respectively as before.
	Our hypothesis implies that  $\rho_{\hat{\sigma}_1}(\hat{\tau}_1)=\rho_{\hat{\sigma}_2}(\hat{\tau}_2)$.
	Thus, we have that the set of entries in the corresponding columns  of $\tau_1$ and $\tau_2$ are the same.
	Since the standardized first column word of both $\tau_1$ and $\tau_2$ is $\sigma$, the first column of $\tau_1$ is the same as the first column of $\tau_2$.
	This implies $\hat{\sigma}_1 = \hat{\sigma}_2$, which, in view of {\cite[Theorem 10 part 2)]{HMR}}, implies that $\hat{\tau}_1=\hat{\tau}_2$.
	It follows immediately that $\tau_1=\tau_2$.
	Thus, $\permutedtoreverse_{\sigma}$ is an injection from $\coprod_{\widetilde{\alpha}=\lambda}\PRCT^{\sigma}(\alpha) \to \RT(\lambda)$.

	To establish the injectivity of $\reversetopermuted_{\sigma}$, we again use results in \cite{HMR}, beginning with an alternative description of $\reversetopermuted_{\sigma}$.
	Let $T\in \RT(\lambda)$.
	Construct a PBF $\hat{T}$ of shape $\hat{\lambda}$ with basement permutation $\overline{\epsilon}_n$ by filling the remaining diagram, essentially the composition diagram of $\lambda$, exactly as $T$.
	To obtain $\reversetopermuted_{\sigma}(T)$, construct a permutation $\hat{\sigma}$ such that
	\begin{itemize}
		\item its first $k$ letters are a rearrangement of the entries in the first column of $T$,
		\item the standardization of the word formed by the first $k$ letters is $\sigma$,
		\item the last $n-k$ entries increase when read from left to right.
	\end{itemize}
	Then $\hat{\tau}\coloneqq\rho_{\hat{\sigma}}^{-1}(\hat{T})=\reversetopermuted_{\hat{\sigma}}(\hat{T})$ is a PBF with basement permutation $\hat{\sigma}$.
	Given how $\hat{\sigma}$ relates to $\sigma$ and to the first column of $T$,
	we infer that the first $k$ rows in the  second column of $\hat{\tau}$ are the same as the first $k$ rows of the first column. 
  This implies that the filling obtained by removing the basement from  $\hat{\tau}$ is in fact {$\reversetopermuted_{\sigma}(T)$}.

	To establish the injectivity of $\reversetopermuted_{\sigma}$, argue as follows.
	Suppose $T_1, T_2\in \RT(\lambda)$  are such that $\reversetopermuted_{\sigma}(T_1)=\reversetopermuted_{\sigma}(T_2)$.
	The preceding discussion implies that there exist permutations $\hat{\sigma}_1, \hat{\sigma_2} \in \sgrp{n}$ such that $\rho_{\hat{\sigma}_1}^{-1}(\hat{T_1})=\rho_{\hat{\sigma}_2}^{-1}(\hat{T_2})$, where $\hat{T_1}$ and $\hat{T_2}$ are obtained by appending the basement permutation $\overline{\epsilon}_n$ as before.
	Again, we conclude that the sets of entries in corresponding columns of $T_1$ and $T_2$ are equal, which implies that $T_1=T_2$.

	To finish the proof, we need to establish that $\reversetopermuted_{\sigma}$ and $\permutedtoreverse_{\sigma}$ are mutually inverse.
	This follows from our alternate description for $\permutedtoreverse_{\sigma}$ and $\reversetopermuted_{\sigma}$ in terms of $\rho_{\hat{\sigma}}$ and $\rho_{\hat{\sigma}}^{-1}$ for an appropriately constructed $\hat{\sigma}$.
\end{proof}
Figure~\ref{fig:HMR} describes the various steps in the proof of Theorem~\ref{thm:generalized shift map} in the case of the $\RT$ and $\PRCT$ from Figure~\ref{fig:prct<->rt}.
Let $\sigma=3142$, $\alpha=(4,1,3,3)$ and $\lambda=(4,3,3,1)$.
The $\PRCT$ $\tau$  on the left belongs to $\PRCT^{\sigma}(\alpha)$.
The second filling  from the left is the PBF $\hat{\tau}$ with basement permutation $\hat{\sigma}$ constructed in the preceding proof.
We  have $\hat{\sigma}=10\ 2 \ 11\ 9\ 1 \ 3 \ 4\ 5 \ 6 \ 7 \ 8$  in one-line notation.
The second filling from the right is  $\rho_{\hat{\sigma}}(\hat{\tau})$ and the rightmost filling is $\permutedtoreverse_{\sigma}(\tau)$.
\begin{figure}[ht]
	\includegraphics{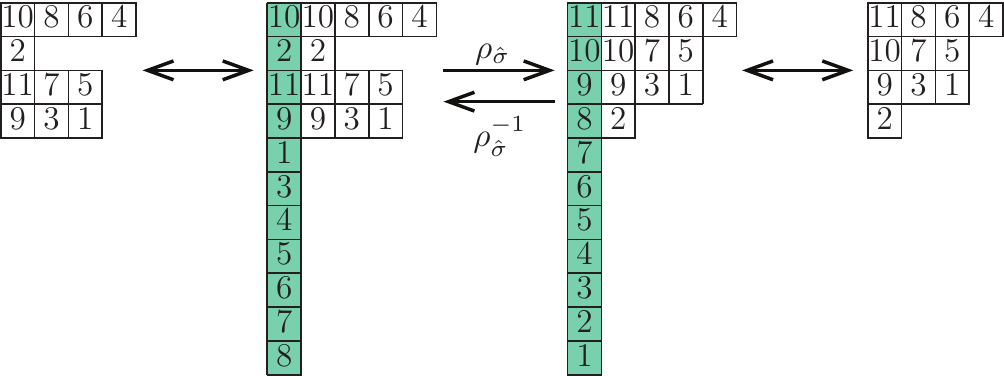}
	\caption{The bijection between $\coprod_{\widetilde{\alpha}=\lambda}\PRCT^{\sigma}(\alpha)$ and $\RT(\lambda)$ via PBFs.}
	\label{fig:HMR}
\end{figure}

We conclude this section by noting that if $\sigma=\epsilon_{\ell(\alpha)}$, then the map $\permutedtoreverse_{\sigma}$ maps a $\RCT$ $\tau$ of shape $\alpha$ to an $\RT$ of shape $\widetilde{\alpha}$.

\section{0-Hecke modules on $\SPRCT$s}\label{sec:0-Hecke}
In order to define an $H_n(0)$-action on $\SPRCT$s we recall the concept of attacking defined in \cite{TvW}. In fact, we employ notions similar to those in \cite[Section 3]{TvW} throughout and essentially all results there are true in our case as well.
We focus solely on deriving analogues relevant for this article.

Given $\rtau \in \SPRCT^{\sigma}(\alpha)$ for some $\alpha\vDash n$, and a positive integer $i$ such that $1\leq i\leq n-1$, we say that $i$ and $i+1$ are \emph{attacking} if either
\begin{enumerate}
	\item $i$ and $i+1$ are in the same column in $\rtau$ (in which case we call them \emph{strongly attacking}), or
	\item $i$ and $i+1$ are in adjacent columns in $\rtau$, with $i+1$ positioned southeast of $i$ (in which case we call them \emph{weakly attacking}).
\end{enumerate}

Given an integer $i$ satisfying $1\leq i\leq n-1$, let $s_i(\rtau)$ denote the filling obtained by interchanging the positions of entries $i$ and $i+1$ in $\rtau$.
Define operators $\pi_i$ for $1\leq i\leq n-1$ as follows.
\begin{eqnarray}\label{eq:pi}
\pi_{i}(\rtau)&=& \left\lbrace\begin{array}{ll}\rtau & i\notin \des(\rtau)\\ 0 & i\in \des(\rtau), i \text{ and }  i+1 \text{ attacking}\\ s_{i}(\rtau) & i\in \des(\rtau), i \text{ and } i+1 \text{ { nonattacking}}\end{array}\right.
\end{eqnarray}
If $i\in \des(\rtau)$ is such that $i$ and $i+1$ are attacking (respectively nonattacking) then $i$ is an \emph{attacking descent} (respectively \emph{nonattacking descent}).

\begin{theorem}\label{the:0heckerels}
	The operators $\{ \pi _i \} _{i=1}^{n-1}$ satisfy the same relations as $\hn$. Thus, given a composition $\alpha\vDash n$, they define a $0$-Hecke action on $\SPRCT^{\sigma}(\alpha)$.
\end{theorem}
The proof of the above theorem in the context of $\SRCT$s  involves a lengthy verification and relies on some preliminary lemmas \cite{TvW}.
The techniques we need vary slightly and we discuss the differences next.
The key lemma for proving that the operators give rise to a $0$-Hecke action on $\SRCT$s is \cite[Lemma 3.7]{TvW} and it, in turn, relies  on \cite[Lemmas 3.4, 3.5, 3.6]{TvW}.
Since we avoid an approach involving box-adding operators here, we cannot use an analogue of \cite[Lemmas 3.4 and 3.5]{TvW} for $\SPRCT$s.
In fact, \cite[Lemma 3.6]{TvW} does not hold in our setting.
In spite of this, we do have the following analogue of  \cite[Lemma 3.7]{TvW}.
\begin{lemma}\label{lem:switchPRCT}
	\begin{enumerate}
		\item If $\alpha \vDash n$, $\rtau \in  \SPRCT^{\sigma}(\alpha)$ and $j\in \des(\rtau)$ such that $j$ and $j+1$ are nonattacking, then $s_j(\rtau)\in \SPRCT^{\sigma}(\alpha)$.
		\item If $\alpha \vDash n$, $\rtau \in  \SPRCT^{\sigma}(\alpha)$ and $j\notin \des(\rtau)$ such that $j$ is not in the cell to the immediate right of $j+1$, then $s_j(\rtau)\in \SPRCT^{\sigma}(\alpha)$.
	\end{enumerate}
\end{lemma}
\begin{proof}
	It is clear in both cases that $s_j(\tau)$ has rows that strictly decrease when read from left to right and that the standardized first column word is $\sigma$.
	Thus, to establish that $s_j(\tau)\in \SPRCT^{\sigma}(\alpha)$, we need to check that the triple condition holds.
	Our proof proceeds by checking the triple condition locally.
	That is, we focus on a triple configuration involving a fixed set of cells in $\tau$ and show that if the triple condition held for this set of cells, then it continues to hold for the same set of cells in $s_j(\tau)$.
	Consider throughout a fixed triple configuration in $\tau$ as in Figure~\ref{fig:triple configuration}.

	Assume now that $j$ is  a nonattacking descent in $\tau$.
	If neither $j$ nor $j+1$ belongs to  $\{a,b,c\}$, then the triple condition obviously holds in $s_j(\tau)$. 
	Hence suppose this is not the case. Our premise implies that exactly one of $j$ or $j+1$ belongs to $\{a,b,c\}$.
	Thus, replacing the $j$ by $j+1$ in $\tau$ does not alter the relative order of entries in the cells in the triple configuration under consideration.
	Therefore, we conclude that $s_j(\tau)\in \SPRCT^{\sigma}(\alpha)$.

	We establish the second part of the lemma using an analysis similar to before.
	Assume that $j\notin \des(\tau)$ such that $j$ is not in the cell to the immediate right of $j+1$.
	As before, we infer that both $j$ and $j+1$ cannot belong to $\{a,b,c\}$. In the case where neither belongs, the triple condition clearly holds in $s_j(\tau)$.
	Hence assume that one of $j$ or $j+1$ belongs to $\{a,b,c\}$.
	Replacing the $j$ by $j+1$ in $\tau$ does not alter the relative order of entries in the cells in the triple configuration under consideration.
	Thus the triple condition holds in $s_j(\tau)$.
\end{proof}

We now state lemmas which show that the relations satisfied by the $0$-Hecke algebra $\hn$ are also satisfied by the operators $\{ \pi _i \} _{i=1} ^{n-1}$.
The proofs are omitted given their similarity to \cite[Lemmas  3.9, 3.10 and 3.11]{TvW} respectively. 
We emphasize here that only the first half of Lemma~\ref{lem:switchPRCT} is needed for establishing the three lemmas that follow, and thereby, Theorem~\ref{the:0heckerels}. The second half of Lemma~\ref{lem:switchPRCT} is crucial towards establishing Lemma~\ref{lem:unique source sink}.

\begin{lemma}\label{lem:pisquared}
	For $1\leq i\leq n-1$, we have $\pi_{i}^{2}=\pi_i$.
\end{lemma}


\begin{lemma}\label{lem:pidifferby2}
	For $1\leq i,j\leq n-1$ such that $\lvert i-j\rvert \geq 2$, we have $\pi_i\pi_j=\pi_j\pi_i$.
\end{lemma}




\begin{lemma}\label{lem:piiiplus1}
	For $1\leq i\leq n-2$, we have $\pi_i\pi_{i+1}\pi_i=\pi_{i+1}\pi_i\pi_{i+1}$.
\end{lemma}
The proof of Theorem~\ref{the:0heckerels}  follows  from Lemmas~\ref{lem:pisquared}, \ref{lem:pidifferby2} and \ref{lem:piiiplus1}.
Let $\smodule_{\alpha}$ denote the $\mathbb{C}$-linear span of all $\SPRCT$s of shape $\alpha$.
We have established that $\smodule_{\alpha}$ is an $\hn$-module.
We now construct a direct sum decomposition of $\smodule_{\alpha}$ by defining an equivalence relation based on standardized column words of $\SPRCT$s.

\subsection{Source and sink tableaux}
Given a composition $\alpha \vDash n$, define an equivalence relation $ \equiva $ on $\SPRCT(\al)$ by defining $\tau_1\equiva \tau_2$ if and only if $\st (\tau_1)= \st(\tau_2)$, that is, the entries in each column of $\tau_1$ are in the same relative order as the entries in the corresponding column of $\tau_2$.
Suppose  that the equivalence classes under $\equiva$ are {$E_1,\ldots, E_k$}.
Let $\smodule_{\alpha,E_i}$ denote the $\mathbb{C}$-linear span of all $\SPRCT$s in $E_i$ for $i=1, \ldots, k$.
Then we get the following isomorphism of vector spaces
\begin{equation}\label{eq:Ssum}
\smodule_{\alpha}\cong\bigoplus_{i=1}^{k}\smodule_{\alpha,E_i},
\end{equation}
which is in fact an $\hn$-module isomorphism as the following lemma implies.
We omit the proof as it is the same as \cite[Lemma 6.6]{TvW}.

\begin{lemma}\label{lem:Emodule}
	Let $E_j$ for $j= 1, \ldots , k$ be the equivalence classes under $\equiva$ for $\alpha \vDash n$. Then
	for all $i$ such that $1\leq i\leq n-1$, we have that $\pi_i(\smodule_{\alpha,E_j}) \subseteq \smodule_{\alpha,E_j}$ for any $1\leq j\leq k$.
\end{lemma}
Next, we discuss two important classes of $\SPRCT$s that form special representatives of each equivalence class.
Let $\alpha\vDash n$. An $\SPRCT$ $\rtau$ of shape $\alpha$ is said to be a \emph{source} tableau if  for every $i\notin \des(\rtau)$ where $i\neq n$, we have that $i+1$ lies to the immediate left of $i$.
An $\SPRCT$ $\rtau$ of shape $\alpha$ is said to be a \emph{sink} tableau if for every $i\in \des(\rtau)$, we have that $i$ and $i+1$ are attacking.
Figure~\ref{fig:source and sink} shows a source tableau (left) and a sink tableau (right) of shape $(1,3,2,4)$ and type $1324$.
Note that the standardized column word of each {tableau} is $1324 \ 213 \ 12 \ 1$.

\begin{figure}[htbp]
	\centering
	\begin{align*}
	\begin{ytableau}
	1\\
	6 & 5 & 4\\
	3 & 2\\
	10 & 9 & 8 & 7
	\end{ytableau} \hspace{5mm}
	\begin{ytableau}
	3\\
	8 & 5 & 1\\
	7 & 4\\
	10 & 9 & 6 & 2
	\end{ytableau}
	\end{align*}
	\caption{An example of a source {tableau} and a sink tableau.}
	\label{fig:source and sink}
\end{figure}
By performing an analysis very similar to that in \cite[Section 6]{TvW} we obtain the following analogue of \cite[Corollary 6.15]{TvW}, which is  pertinent for our purposes.
See Remark~\ref{rem:drn} for a brief discussion regarding its proof.
\begin{lemma}\label{lem:unique source sink}
	There is a unique source {tableau} and a unique sink {tableau} in every equivalence class under $\equiva$.
\end{lemma}

\begin{remark}\label{rem:drn}
	The {proofs} of the various lemmas leading up to the proof of \cite[Corollary 6.15]{TvW} {employ} the notion of removable nodes of a composition diagram followed by distinguished removable nodes in $\SRCT$s.
	The definition of removable nodes in the aforementioned case is inspired by the cover relation for the (left) composition poset {\cite[Section 2.2]{BTvW}}.
	Saturated chains in this poset are in bijection with $\SRCT$s.
	Although we have not described $\SPRCT$s as saturated chains in a poset on compositions, we may still define a removable node as follows.
	Given the composition diagram $\al$, we call a cell in position $(i,j)$ a \emph{removable node} if one of the two conditions below is met.
	\begin{enumerate}
		\item $\alpha_i=j\geq 2$ and there is no part of length $j-1$ strictly north of $\alpha_i$.
		\item $\alpha_i=j=1$.
	\end{enumerate}
	Figure~\ref{fig:removable nodes} shows $\alpha=(2,3,1,4)$ with removable nodes corresponding to cells filled with bullets.
	\begin{figure}[htbp]
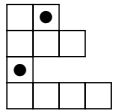

		\centering
		\begin{align*}
		{\ytableausetup{mathmode,boxsize=0.8em}
		\begin{ytableau}
		*(white) & \bullet\\
		*(white) & *(white) &*(white)\\
		\bullet\\
		*(white) & *(white) &*(white) & *(white)
		\end{ytableau}}
		\end{align*}
		\caption{The removable nodes of $\al=(2,3,1,4)$.}
		\label{fig:removable nodes}
	\end{figure}
	With this definition of removable node, one can define distinguished removable nodes in precisely the same way as in \cite[Section 6]{TvW}.
	In particular, the analogues of \cite[Propositions 6.13 and 6.14]{TvW}, which are crucial for deriving our Lemma~\ref{lem:unique source sink}, hold in our setting.
	Thus concludes our remark.

\end{remark}

Recall that our motivation  as outlined in the introduction is to enumerate compatible pairs of permutations.
Consider a $\RCT$ $\tau$ of shape $\alpha$ whose $i$-th and $i+1$-th column, for some $i\geq 1$ contain the same number of entries, say $n$.
Then $(\st_i(\tau), \st_{i+1}(\tau))$ is a compatible pair of permutations in $\sgrp{n}$.
Note that if we consider the entries in columns $i$ and $i+1$ only, we can construct a filling of shape $(2^n)$.
As the triple condition is clearly satisfied, this filling is in fact a $\PRCT$.
We infer that compatible pairs of permutations in $\sgrp{n}$ are a subset of the set of pairs of permutations $(\st_1(\tau), \st_2(\tau))$ corresponding to {tableaux}  $\tau \in\PRCT((2^n))$.
In Section~\ref{sec:allowable acyclic}, we establish that this containment is in fact an equality.
We proceed towards studying $\SPRCT((2^n))$ in depth.

\section{Descents in $2$-columned tableaux and ascent-descents on labeled binary trees}\label{sec:ascents-descents-trees}
We refer to $\SPRCT$s of shape $(2^n)$ for some $n\geq 1$ as \emph{$2$-columned tableaux}  of \emph{size} $2n$.
Consider the following refined classification of certain descents in a $2$-columned tableau $\tau$.
Suppose $i$ is an entry in the first column of $\tau$.
Then $i\in \des(\tau)$.
Based on the relative position of $i+1$ in $\tau$, we may define the following sets.
\begin{align*}
\des_N(\tau)&=\{i\suchthat i+1 \text{ is in the first column north of }i\}\\
\des_S(\tau)&=\{i\suchthat i+1 \text{ is in the first column south of }i\}\\
\des_{NE}(\tau)&=\{i\suchthat i+1 \text{ is in the second column northeast of }i\}\\
\des_{SE}(\tau)&=\{i\suchthat i+1 \text{ is in the second column southeast of }{i\}}
\end{align*}
Additionally, we say that $i$ an $N$-descent (respectively $S$-descent) of $\tau$ if $i+1$ is in the first column and north (respectively south) of $i$.
Similarly, we say that $i$ an $NE$-descent (respectively $SE$-descent) of $\tau$ if $i+1$ is in the second column and northeast (respectively southeast) of $i$.
Figure~\ref{fig:descents in tab} shows a $\tau\in \SPRCT((2^n))$ where $n=10$.
The arcs on the left connect $i$ and $i+1$ where both belong to the first column.
In particular, the number of red arcs on the left is equal to  $|\des_{N}(\tau)|$ whereas the number of purple arcs on the left is equal to  $|\des_{S}(\tau)|$.
The arcs on the right connect rows where the $i$ is in the first column and the $i+1$ is in the second column.
In particular, the number of red arcs on the right is equal to  $|\des_{NE}(\tau)|$ whereas the number of purple arcs on the left is equal to  $|\des_{SE}(\tau)|$.

\begin{figure}[htbp]
	\includegraphics[scale=0.80]{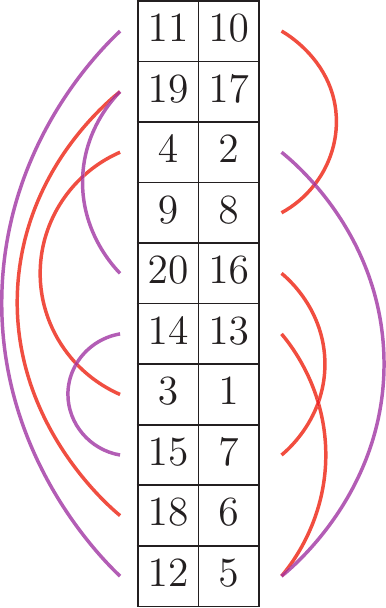}
	\caption{The classification of descents in the first column of {$\tau$.}}
	\label{fig:descents in tab}
\end{figure}

To enumerate ordered pairs $(\st_1(\tau),\st_2(\tau))$ where $\tau$ ranges over $2$-columned tableaux of size $2n$, we need to enumerate  equivalence classes under $\sim_{(2^n)}$.
By Lemma~\ref{lem:unique source sink}, we may equivalently enumerate sink tableaux of shape $(2^n)$.
Such a sink $\tau$  can be alternatively characterized as one
that satisfies $|\des_{NE}(\tau)|=0$.
This motivates our study of the distribution of the quadruple of statistics $$(|\des_N(\tau)|,|\des_S(\tau)|,|\des_{NE}(\tau)|,|\des_{SE}(\tau)|)$$
as $\tau$ varies over all tableaux in $\SPRCT((2^n))$.
We establish that this distribution coincides with the distribution of left ascents, left descents, right ascents and right descents over the set of labeled binary trees on $n$ nodes.
To this end, we need certain combinatorial notions attached to  well-known `Catalan objects' known as Dyck paths.

\subsection{Dyck paths, unlabeled and labeled}\label{subsec:dyck-ldyck}
Given a nonnegative integer $n$, a \emph{Dyck path} of \emph{semi-length} $n$ is a lattice path in the plane beginning at $(0,0)$ and ending at $(2n,0)$ consisting of \emph{up-steps}, which correspond to a translation by {$(1,1)$}, and \emph{down-steps}, which correspond to a translation by $(1,-1)$.
We denote the set of Dyck paths of semi-length $n$ by $\dyck_n$.
The cardinality of $\dyck_n$ is well-known to be ${\Cat}_{n}$, the $n$-th Catalan number equaling $\frac{1}{n+1}\binom{2n}{n}$.
We can identify a Dyck path $D\in \dyck_n$ with a ud-word of length $2n$ by writing  $\upstep$ (respectively $\downstep$) for every up-step (respectively down-step) in $D$ from left to right.
We denote this word by $\dword(D)$ and refer to it as the \emph{Dyck word} of $D$.
A Dyck path is \emph{prime} if it touches the $x$-axis at its initial and terminal points only.
Every Dyck path $D$ can be factorized uniquely as a concatenation of prime Dyck paths $D_1\cdot D_2\cdots D_k$.
We call each $D_i$ a \emph{factor} of $D$.
Figure~\ref{fig:dyck} depicts an element $D\in \dyck_9$ with $4$ factors. We have $\dword(D)=\upstep\downstep \ \upstep\downstep \ \upstep\upstep\upstep\downstep\upstep\downstep\downstep\upstep\downstep\downstep \ \upstep\upstep\downstep\downstep$, where the spaces separate the Dyck words of the factors of $D$.
\begin{figure}[htbp]
	\includegraphics[scale=0.75]{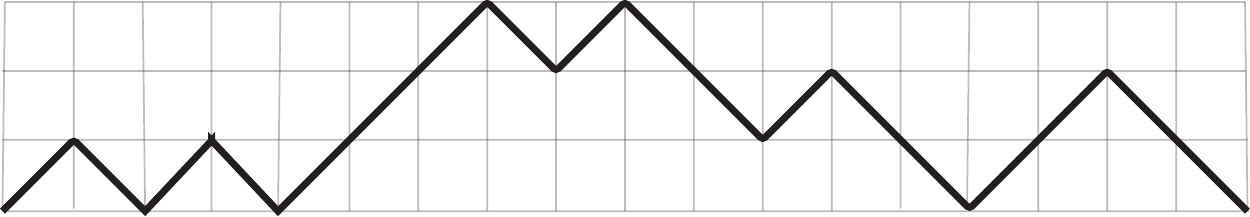}
	\caption{A Dyck path $D$ of semi-length $9$.}
	\label{fig:dyck}
\end{figure}

Consider $D\in \dyck_n$.
If $D$ is endowed with a labeling on its down-steps with distinct positive integers drawn from $[n]$, then we obtain a \emph{labeled Dyck path} $D'$ of semi-length $n$.
We denote the set of labeled Dyck paths of semi-length $n$ by $\ldyck_n$.
The cardinality of $\ldyck_n$ equals $n!\Cat_{n}$.
We define prime labeled Dyck paths and factors of labeled Dyck paths in the obvious manner following the description in the unlabeled case.
Given $D\in \ldyck_n$, we abuse notation and denote by $\dword(D)$ the Dyck word of the underlying unlabeled Dyck path.
Constructing the \emph{labeled Dyck word} $\ldword(D)$ is slightly more involved and is described next.
First we recursively assign labels to the up-steps from right to left.
Consider the $i$-th up-step in $D$ for $i$ from $n$ down to $1$.
Let $\mathfrak{D}_i$ (respectively $\mathfrak{U}_i$) be the set of labels that belong to down-steps (respectively up-steps) after the $i$-th up-step.
Clearly $\mathfrak{U}_{n}=\emptyset$.
Assign the minimum of $\mathfrak{D}_i\setminus \mathfrak{U}_i$ as the label of  this $i$-th up-step.
Note that as the algorithm executes, we always have that $\mathfrak{U}_i$ is a strict subset of $\mathfrak{D}_i$.
Here we are crucially using the fact that $D$ is a (labeled) Dyck path.
Thus, the minimum of $\mathfrak{D}_i\setminus \mathfrak{U}_i$
is well-defined and our procedure terminates when $i$ attains the value $1$.
Next, read the resulting Dyck path, with labels on both up-steps and down-steps, from left to right and write ${\upstep}_{i}$ (respectively ${\downstep}_{i}$) for an up-step (respectively down-step) labeled $i$.
The word thus obtained is the labeled Dyck word $\ldword(D)$.

It is easy to see that for a labeled Dyck path $D$ with prime factorization $D_1\cdots D_k$, we have $$\ldword(D)=\ldword(D_1)\cdots \ldword(D_k).$$
Figure~\ref{fig:ldyck} shows a labeled Dyck path $D$ of semi-length $10$.
The arcs join the up-steps and down-steps that share the same label during the procedure for computing $\ldword(D)$.
We have $$\ldword(D)=
{\upstep}_{7}{\upstep}_{3}{\downstep}_{7}{\downstep}_{3}\hspace{3mm}{\upstep}_{10}{\upstep}_{9}{\upstep}_{8}{\upstep}_{4}{\downstep}_{4}{\upstep}_{1}{\downstep}_{1}{\downstep}_{10}{\upstep}_{6}{\downstep}_{6}{\downstep}_{8}{\upstep}_{5}{{\upstep}_{2}}{\downstep}_{9}{\downstep}_{2}{\downstep}_{5}.$$
\begin{figure}[htbp]
	\includegraphics[scale=0.90]{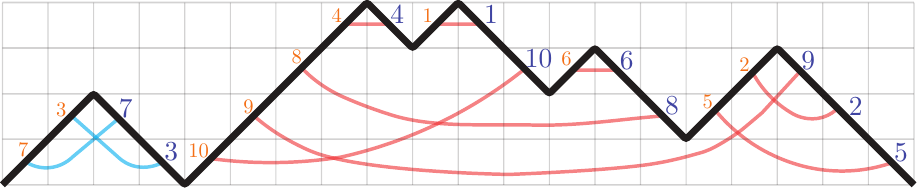}
	\caption{A labeled Dyck path of semi-length $10$.}
	\label{fig:ldyck}
\end{figure}

\subsection{Bijection between $\ldyck_n$ and $\SPRCT((2^n))$}
\label{subsec:ldyck-prct}
Recall  that $|\ldyck_n|=n!\Cat(n)$.
By Theorem~\ref{thm:generalized shift map}, we have that  $|\SPRCT^{\sigma} ((2^n))|={|\SRT((2^n))|}$ for all  $\sigma\in \sgrp{n}$.
The following bijection between $\SRT((2^n))$ and $\dyck_n$ is folklore.
For $T\in \SRT((2^n))$, the $i$-th step for $1\leq i\leq 2n$ in the associated Dyck path is an up-step if $i$ is in the second column of $T$ and it is a down-step otherwise.
Thus, we infer that $|\SPRCT((2^n))|=|\coprod_{\sigma\in \sgrp{n}} \SPRCT^{\sigma}((2^n))|=\sum_{\sigma\in \sgrp{n}} |\SRT((2^n))|=n!\Cat_n$.

By a simple extension of the aforementioned bijection, we obtain one between $\ldyck_n$ and $\SPRCT((2^n))$ {that} associates a  $D\in \ldyck_n$ with $\tau\in \SPRCT((2^n))$ as follows.
For $1\leq i\leq 2n$, the $i$-th step is an up-step if $i$ belongs to the second column.
Otherwise, the $i$-th step is a down-step labeled by the index of the row occupied by $i$ in $\tau$.
We denote this map from $\SPRCT((2^n))$ to $\ldyck_n$ by $\prcttoldyck$ and defer the proof that this is indeed a bijection to after we describe the inverse map, denoted by $\ldycktoprct$, next.

\begin{figure}[htbp]
	\includegraphics[scale=0.9]{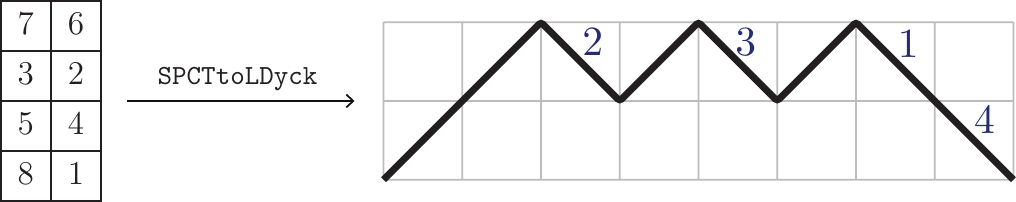}
	\caption{Transforming an SPCT to a labeled Dyck path.}
	\label{fig:spcttoldyck}
\end{figure}

Let $D\in \ldyck_n$ with $\ldword(D)=w_1\cdots w_{2n}$.
Construct a filling $\tau$ of shape $(2^n) $ as follows.
For $1\leq i\leq n$, let $1\leq p<q\leq 2n$ be the positive integers such that $ w_p={\upstep}_{i}$ and $w_q={\downstep}_{i}$.
Then the $i$-th row of $\tau$ contains $q$ and $p$ in the first and  second columns respectively.
We define $\tau$ to be $\ldycktoprct(D)$.
The next lemma establishes  that $\ldycktoprct$ is the {inverse of} $\prcttoldyck$.

\begin{lemma}\label{lem: prct_to_ldyck bijection}
	\emph{$\ldycktoprct$} is a bijection from  $\ldyck_n$ to $\SPRCT((2^n))$ with inverse given by \emph{$\prcttoldyck$}.
\end{lemma}
\begin{proof}
	Consider $D\in \ldyck_n$ and let $\tau=\ldycktoprct(D)$.
	To establish that $\tau$ is an $\SPRCT$, we must show that the triple condition holds.
	Suppose that this is not the case.
	Then we have a configuration in $\tau$ of the type shown in Figure~\ref{fig:triple violation} such that $a>c>b$.
	\begin{figure}[h]
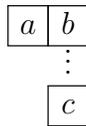

		\centering
		\begin{align*}
		\ytableausetup{mathmode,boxsize=1.25em}
		\begin{ytableau}
		a & b\\
		\none & \none[\vdots]\\
		\none & c
		\end{ytableau}
		\end{align*}
		\caption{A triple configuration in $\ldycktoprct(D)$.}
		\label{fig:triple violation}
	\end{figure}
	We interpret this configuration in terms of $D$.
	Henceforth, in this proof, we assume that the up-steps in $D$ are  labeled using the procedure for computing $\ldword(D)$.

	Assume that $a$ and $b$ belong to row $i$ and that $c$ belongs to row $j$, where $1\leq i<j\leq n$.
	Then  the $a$-th step (respectively $b$-th step) of $D$ is a down-step (respectively up-step) labeled $i$.
	Let $S_b$ and $S_c$ be the sets consisting of labels on the down-steps after the $b$-th and $c$-th steps respectively, ordered in increasing order.
	We have  $S_c\subseteq S_b$ as $c>b$ and  $i\in S_c$ as $a>c$.
	As the $c$-th step, which is an up-step, gets the label $j$, and $j>i$, we conclude that there exists an up-step after the $c$-th step that has the label $i$.
	This follows from the description of our procedure for constructing the labeled Dyck word.
	Thus, we arrive at a contradiction as the $b$-th step cannot be labeled $i$.
	We conclude that the triple condition holds in $\tau$ and that $\tau\in \SPRCT((2^n))$.

	To show that $\ldycktoprct$ is a bijection, we establish its injectivity.
	If $\ldycktoprct(D_1)=\ldycktoprct(D_2)$ for labeled Dyck paths $D_1,D_2$, then we must have $\ldword(D_1)=\ldword(D_2)$.
	This implies that the unlabeled Dyck paths underlying  $D_1$ and $D_2$ are the same. Additionally, the labels on the corresponding down-steps in $D_1$ and $D_2$ are equal.
	Thus, $D_1=D_2$ and we conclude that $\ldycktoprct$ is {a bijection}.

	It is clear that $\prcttoldyck(\tau)=D$ as the entries in the first column of $\tau$ and the indices of the rows they belong to completely determine the down-steps and their labels.
	We conclude that $\prcttoldyck$ and $\ldycktoprct$ are  bijections that are mutually inverse.
\end{proof}
We proceed to describe unlabeled and labeled binary trees and subsequently introduce  descent statistics that are equidistributed with the quadruple $(\des_N,\des_S,\des_{NE},\des_{SE})$.

\subsection{Plane binary trees, unlabeled and labeled}
A \emph{rooted tree} is a finite acyclic graph with a distinguished node called the \emph{root}.
A \emph{plane binary tree} is a rooted tree in which every node has at most two children, of which at most one is called a \emph{left child} and at most one is called a \emph{right child}.
Henceforth, we reserve the term binary tree for a plane binary tree.
We denote the set of binary trees on $n\geq 1$ nodes by $\tree_n$.
The set of nodes of $T\in \tree_n$ is denoted by $V(T)$, and the root node  is referred to as $\rootof{T}$.
We abuse notation on occasion and write $v\in T$ when we mean $v\in V(T)$.
A node $v\in T$ is an \emph{internal node} if it has at least one child.
Otherwise it is a \emph{leaf}.

A \emph{\lpath} is a binary tree $T$ such that no node has a right child.
We can decompose any binary tree $T$ into a set of \lpaths $\{T_1,\ldots,T_k\}$ by omitting all edges that connect an internal node to its right child, provided it exists.
We refer to this set of \lpaths as the \emph{maximal left path  decomposition} of $T$ and denote it by $\mlpd(T)$.
Figure~\ref{fig:Tree_decomposition} depicts a  binary tree $T$ on $10$ nodes with $V(T)={\{v_1,\ldots, v_{10}\}}$.
The nodes that belong to the same shaded region constitute a \lpath in $\mlpd(T)$.

Given $T\in \tree_n$, let $\mlpd(T)$ be {$\{T_1,\ldots, T_m\}$} where $T_1$ contains $\rootof{T}$.
For every $T_i$ where $2\leq i\leq m$,  there exists a unique node $v\in T$ such that $\rootof{T_i}$ is the right child of $v$ in $T$.
We define $v$ to be the \emph{parent} of $T_i$ in $T$ and denote it by $\parent(T_i)$.
In Figure~\ref{fig:Tree_decomposition}, consider the \lpath  on nodes $v_{7}$ and $v_{8}$. Then its parent is $v_6$.
Note the crucial fact that a binary tree $T$ is completely determined by specifying its maximal left path decomposition, the \lpath that contains the root, and the parents of the other left paths.

\begin{figure}[htbp]
	\includegraphics[scale=1.1]{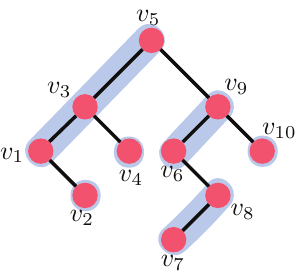}
	\caption{A binary tree with its maximal \lpath decomposition.}
	\label{fig:Tree_decomposition}
\end{figure}

A \emph{labeled plane binary tree} (or simply a \emph{labeled binary tree}) on $n\geq 1$ nodes is a binary tree whose nodes have distinct labels drawn from $[n]$.
We denote the set of labeled binary trees on $n$ nodes for $n\geq 1$ by $\ltree_n$.
Given a node $u\in T$, we refer to its label as $u^{\ell}$.
For a labeled binary tree, we have the following refined classification for its edges.
Suppose that $q$ is the right child of $p$.
If $p^{\ell}< q^{\ell}$, then the edge joining $p$ and $q$ is a \emph{right ascent}.  Otherwise it is a \emph{right descent}.
Now suppose that  $q$ is the left child of $p$.
If $p^{\ell}< q^{\ell}$, then  the edge joining $p$ and $q$ is a \emph{left ascent}. Otherwise it is a \emph{left descent}.
Let $\rasc(T)$ (respectively, $\rdes(T)$, $\lasce(T)$ and  $\ldes(T)$) {be the} number of right ascents (respectively, right descents, left ascents and left descents) in $T$.
For the labeled binary tree $T$ in Figure~\ref{fig:Ascents-Descents},  we have that $\rasc(T)=3$, $\rdes(T)=1$, $\lasce(T)=1$, $\ldes(T)=3$. The thickened edges in the figure correspond to descents, both left and right.

\begin{figure}[ht]
	\includegraphics{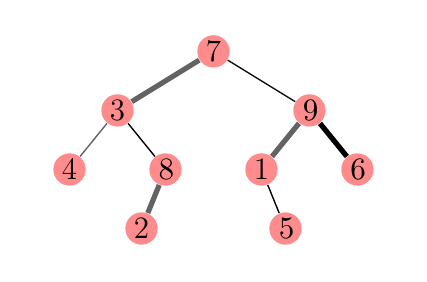}
	\caption{A labeled binary tree with $\rasc(T)\!=\!\!3$, $\rdes(T)\!=\!1$, $\lasce(T)\!=\!1$, $\ldes(T)\!=\!3$.}
	\label{fig:Ascents-Descents}
\end{figure}

We will abuse notation and use the terms  \lpath and maximal \lpath decomposition in the context of their natural labeled analogues respectively.

\subsection{Labeled Dyck paths and labeled binary trees}
\label{subsec:ltree<->ldyck}
Note that since $|\tree _n|=\Cat _n$ the cardinalities of $\ldyck_n$ and $\ltree_n$ are equal.
We seek a bijection between these two sets that tracks the quadruple of statistics $(\lasce,\ldes,\rasc,\rdes)$ on labeled binary trees in a manner that will eventually allow us to relate them to the quadruple {$(|\des_N|,|\des_S|,|\des_{NE}|,|\des_{SE}|)$} via the map $\ldycktoprct$.
First, we need some more terminology pertaining to Dyck paths.

A \emph{run} of $D\in \dyck_n$ is a maximal sequence of down-steps. A \emph{run} of $D\in \ldyck_n$ is defined similarly, except that we retain the labels.
By considering all runs in  $D\in \ldyck_n$ from right to left, we obtain the \emph{run sequence} of $D$.
Consider a run in $D$ composed of $a$ steps.
We can transform this run into a \lpath on $a$ nodes where the label of the nodes read from root to leaf are the labels in the run read from right to left.
Thus, from the run sequence of $D$, say $(R_1,\dots, R_m)$, we obtain a sequence of \lpaths $(T_1,\dots, T_m)$.
This datum can be used to construct a labeled binary tree as described in Algorithm~\ref{alg:ldycktoltree}.
As usual, by the label on an up-step of a labeled Dyck path, we mean the label it acquires when computing its labeled Dyck word.

\begin{algorithm}[htbp]
	\caption{{Transforming} a labeled Dyck path to a labeled binary tree \label{alg:ldycktoltree}}
	\begin{algorithmic}[1]
		\Require{$D\in\ldyck_n$.}
		\Ensure{$T\in\ltree_n$.}
		\vspace{5pt}
		\Function{ \ldycktoltree}{$D$}
		\State \parbox[t]{\dimexpr\linewidth-\algorithmicindent}
		{Let $(R_1,\ldots,R_m)$ be the run sequence of {$D$;}\strut}
		\State \parbox[t]{\dimexpr\linewidth-\algorithmicindent}
		{Let $(T_1,\ldots, T_m)$ be the sequence of \lpaths corresponding to the run sequence;\strut}
		\State Let $T\coloneqq T_1$, $i\coloneqq 2$;
		\While{$ i < m$}
		\State \parbox[t]{\dimexpr\linewidth-\algorithmicindent}{$j\coloneqq $ label on the up-step immediately after the rightmost down-step in $R_i$;\strut}
		\State \parbox[t]{\dimexpr\linewidth-\algorithmicindent}
		{Make the \lpath $T_i$ the right subtree of the node labeled $j$ in $T$;\strut}
		\State Let $T$ be this updated tree;
		\State $i:=i+1$;
		\EndWhile
		\State\Return{$T$};
		\EndFunction
	\end{algorithmic}
\end{algorithm}

Figure~\ref{fig:ldyck->ltree} demonstrates how Algorithm~\ref{alg:ldycktoltree} transforms the labeled Dyck path in Figure~\ref{fig:ldyck} into a labeled tree.
\begin{figure}[htbp]
	\includegraphics{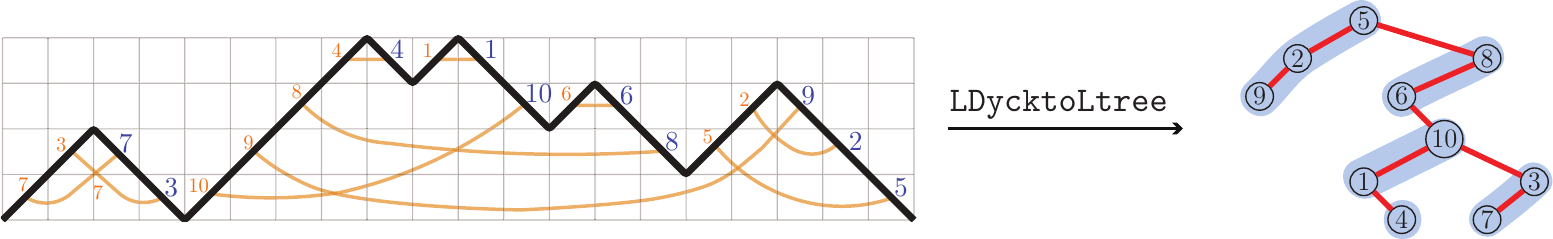}
	\caption{Transforming a labeled Dyck path to a labeled binary tree.}
	\label{fig:ldyck->ltree}
\end{figure}

Note that the run sequence determines the maximal labeled left path decomposition of the resulting binary tree.
To determine its structure completely, we need to designate which \lpath  contains the root of the final binary tree and subsequently determine the parent nodes of the other left paths.
The procedure for computing  the labeled Dyck word accomplishes the latter goal.
The proof that $\ldycktoltree$ is a bijection is postponed until after discussing the inverse map, which is slightly more involved.
We note some key features of Algorithm~\ref{alg:ldycktoltree}. The details, which are straightforward, are left to the reader.

\begin{lemma}\label{lem:two features ldycktoltree}
	Algorithm~\ref{alg:ldycktoltree} posseses the following  features.
	\begin{enumerate}
		\item If $\emph{\ldword}(D)$ has last letter equal to ${\downstep}_i$ for some positive integer $i$, then the root of $\emph{\ldycktoltree}(D)$ has label $i$.
		\item If $\emph{\ldword}(D)$ has a contiguous subword of the type ${\downstep}_{i}{\downstep}_{j}$ for  positive integers $i$ and $j$, then $i$ is the label of the left child of the  node labeled $j$ in $\emph{\ldycktoltree}(D)$.
		\item If $\emph{\ldword}(D)$ has a contiguous subword of the type ${\downstep}_{i}{\upstep}_{j}$ for  positive integers $i$ and $j$, then $i$ is the label of the right child of the node labeled $j$ in $\emph{\ldycktoltree}(D)$.
	\end{enumerate}
\end{lemma}

In Algorithm~\ref{alg:ltreetoldyck} we describe a procedure that converts  $T\in \ltree_n$ to the labeled Dyck word of some $D\in \ldyck_n$.
\begin{algorithm}[htbp]
	\caption{{Transforming} a labeled binary tree to a labeled Dyck path\label{alg:ltreetoldyck}}
	\begin{algorithmic}[1]
		\Require{$T\in\ltree_n$.}
		\Ensure{A labeled Dyck word corresponding to a labeled Dyck path in $\ldyck_n$.}
		\vspace{5pt}
		\Function{ \ltreetoldyck}{$T$}
		\State \parbox[t]{\dimexpr\linewidth-\algorithmicindent}{Let $Q\coloneqq \emptyset$, $\DyckWord\coloneqq \varepsilon$, $\NumSteps\coloneqq 0$, {$\mathrm{flag}\coloneqq 1$;}\strut}
		\State \parbox[t]{\dimexpr\linewidth-\algorithmicindent}{Let $C$ be the unique \lpath  in $\mlpd(T)$ containing $\rootof{T}$;\strut}
		\While{$\NumSteps < 2n$}
		\If{$\mathrm{flag}$ is equal to 1}
		\State Let $v_1,\dots,v_m$ be the nodes in $C$ from root to leaf;
		\For{$1\leq j \leq m$}
		\State $Q\coloneqq Q\cup \{v_j\}$; \Comment{Push $i$ into $Q$}
		\State $i\coloneqq $label of node $v_j$;
		\State $\DyckWord\coloneqq {\downstep}_{i}\cdot\DyckWord$;
		\State $\NumSteps\coloneqq \NumSteps+1$;
		\EndFor
		\State $\mathrm{flag}\coloneqq 0$;
		\Else{}
		\State Let $v$ be the node in ${Q}$ whose label $m$ in minimal;
		\State $Q\coloneqq Q\setminus\{v\}$; \Comment{Pop minimum element from $Q$}
		\State $\DyckWord\coloneqq {\upstep}_{m}\cdot \DyckWord$;
		\State $\NumSteps\coloneqq \NumSteps+1$;
		\State Let $v$ be node in $T$ with label $m$;
		\If{$v$ has a right child $v'$ in $T$}
		\State Let $C$ be the unique \lpath in $\mlpd(T)$ rooted at {$v'$;}
		\State $\mathrm{flag}\coloneqq 1$;
		\EndIf
		\EndIf
		\EndWhile
		\State\Return{$\DyckWord$};
		\EndFunction
	\end{algorithmic}
\end{algorithm}
We refer to the sequence of Steps 8, 9, 10 and 11 in Algorithm~\ref{alg:ltreetoldyck} as a (single) \emph{push} operation. The sequence of Steps 14, 15, 16 and 17 comprises a (single) \emph{pop} operation.

\begin{table}[htbp]
	\begin{center}
		\begin{tabular}{| l | l | l |}
			\hline
			$\NumSteps$ &  $Q$ &  Operation\\
			\hline
			$1$ &  $\{5\}$ & Push $5$ into $Q$\\
			\hline
			$2$ &  $\{2,5\}$ & Push $2$ into $Q$\\
			\hline
			$3$ &  $\{2,5,9\}$ & Push $9$ into $Q$ \\
			\hline
			$4$ &  $\{5,9\}$ & Pop $2$ from $Q$\\
			\hline
			$5$ &  $\{9\}$ &  Pop $5$ from $Q$\\
			\hline
			$6$ &  $\{8,9\}$ & Push $8$ into $Q$\\
			\hline
			$7$ &  $\{6,8,9\}$ & Push $6$ into $Q$\\
			\hline
			$8$ &  $\{8,9\}$ & Pop $6$ from $Q$\\
			\hline
			$9$ &  $\{8,9,10\}$ & Push $10$ into $Q$\\
			\hline
			$10$ &  $\{1,8,9,10\}$ & Push $1$ into $Q$ \\
			\hline
			$11$ &  $\{8,9,10\}$ & Pop $1$ from $Q$\\
			\hline
			$12$ &  $\{4,8,9,10\}$ & Push $4$ into $Q$\\
			\hline
			$13$ &  $\{8,9,10\}$ &  Pop $4$ from $Q$\\
			\hline
			$14$ &  $\{9,10\}$ & Pop $8$ from $Q$\\
			\hline
			$15$ &  $\{10\}$ & Pop $9$ from $Q$\\
			\hline
			$16$ &  $\emptyset$ & Pop $10$ from $Q$\\
			\hline
			$17$ &  $\{3\}$ & Push $3$ into $Q$\\
			\hline
			$18$ &  $\{3,7\}$ & Push $7$ into $Q$\\
			\hline
			$19$ &  $\{7\}$ & Pop $3$ from $Q$\\
			\hline$20$ &  $\emptyset$ & Pop $7$ from $Q$\\
			\hline
		\end{tabular}
	\end{center}
	\caption{Algorithm~\ref{alg:ltreetoldyck} applied to the labeled binary tree in Figure~\ref{fig:ldyck->ltree}.}
	\label{tab:push-pop}
\end{table}
Table~\ref{tab:push-pop} demonstrates the execution of Algorithm~\ref{alg:ltreetoldyck} with the labeled binary tree of Figure~\ref{fig:ldyck->ltree} as input.
Reading the sequence of pushes and pop from bottom to top in this table and noting a $\downstep$ for a {push} and a $\upstep$ for a {pop}, along with the appropriate subscript,  we obtain the following $\DyckWord$ as final output.
Note that this is the labeled Dyck word corresponding to the {labeled} Dyck path in Figure~\ref{fig:ldyck->ltree}.
{$$
{\upstep}_{7}{\upstep}_{3}{\downstep}_{7}{\downstep}_{3}{\upstep}_{10}{\upstep}_{9}{\upstep}_{8}{\upstep}_{4}{\downstep}_{4}{{\upstep}_{1}}{\downstep}_{1}{\downstep}_{10}{\upstep}_{6}{\downstep}_{6}{\downstep}_{8}{\upstep}_{5}{\upstep}_{2}{\downstep}_{9}{\downstep}_{2}{\downstep}_{5}
$$}
For the rest of this section, we adhere to the notation introduced in Algorithm~\ref{alg:ltreetoldyck}.
Our next lemma shows that this algorithm is well-defined.
In particular, it establishes that a pop is never performed on an empty $Q$ as the algorithm executes.
The acyclicity of a tree guarantees that any node can be pushed at most once during the execution of Algorithm~\ref{alg:ltreetoldyck}.
Furthermore, once it is pushed, it must take part in exactly one pop operation.
In fact, the following stronger claim holds: Every node gets pushed and popped exactly once during the execution of our algorithm.
This observation, which hinges on the fact that $T$ is connected, is present implicitly in the proof of the next lemma.

\begin{lemma}\label{lem:validity of algorithm}
	In Algorithm~\ref{alg:ltreetoldyck}, there can never be a pop operation when $Q$ is empty.
\end{lemma}
\begin{proof}
	Suppose $Q$ becomes empty at some point during our algorithm upon popping
	the node $v$.
	Let $T_1\in \mlpd(T)$ be the \lpath that contains $v$.
	As $v$ is the last node to be popped from $Q$, we conclude that every other node in $T_1$ except $v$ has been popped from $Q$ before $v$.
	If $v$ has a right child $v_0$ in $T$, then the next step of our algorithm, after setting  flag  equal to $1$, pushes all nodes that belong to the  \lpath in $\mlpd(T)$ containing $v_0$ into {$Q$.
	This} ensures that the next  pop is performed on a nonempty $Q$.

	Assume henceforth that $v$ does not have a right child in $T$.
	We claim that $\NumSteps$ equals $2n$ upon performing the pop that removes $v$ from $Q$, thereby implying that the algorithm terminates subsequently.
	To establish this claim, we proceed by contradiction.
	Suppose $\NumSteps$ is strictly less than $2n$ upon popping $v$.
	Then there exists a node $v' \in T$ that has not been pushed into $Q$ yet.
	For if there was no such node, {then every} node in $T$ would have been pushed and popped from $Q${, since $Q$ becomes empty upon popping $v$}, contradicting our premise that $\NumSteps<2n$.

	Let  $T_1'\in \mlpd(T)$ be the \lpath containing $v'$.
	As $v'$ has not been pushed into $Q$ yet, we infer that no node in  $T_1'$ has been pushed into $Q$.
	In particular, $\rootof{T}$ does not belong to $T_1'$ and $T_1'\neq T_1$.
	Let $T_2'\in \mlpd(T)$ be the \lpath that contains  $\parent(T_1')$.
	As no node of $T_1'$ has been pushed into $Q$, we infer that no element in $T_2'$ has been pushed into $Q$ either.
	As before, we have that $\rootof{T}$ is not in $T_2'$ and $T_2'\neq T_1$.
	Continuing this inductively we obtain an infinite sequence of distinct \lpaths $T_1',T_2',\dots$ such that $T_{i+1}'\in \mlpd(T)$ contains $\parent(T_i')$ for $i\geq 1$.
	Each of these \lpaths contains nodes that have not been pushed into $Q$.
	As $V(T)$ is finite, we have a contradiction.
\end{proof}

We note an important consequence of Lemma~\ref{lem:validity of algorithm}.
At every stage during the execution of Algorithm~\ref{alg:ltreetoldyck}, the number of pushes weakly exceeds the number of pops.
Additionally, every time a pop empties $Q$, the number of pushes equals the number of pops up until that stage.
Conversely, if the number of pops equals the number of pushes at any stage, then $Q$ must be empty.
These observations imply that in every suffix of $\DyckWord$, there are at least many $\downstep$'s as there are $\upstep$'s.
In summary, we have the following.

\begin{corollary}\label{cor:output is a dyck path}
	$\chi(\DyckWord)$ is the Dyck word of a Dyck path.
\end{corollary}
Our next aim is to establish the stronger assertion that $\DyckWord$ is the labeled Dyck word of some labeled Dyck path.
Let $D^u$ denote the Dyck path corresponding to $\chi(\DyckWord)$.
Let $D^{\ell}$ denote the labeled Dyck path whose underlying unlabeled Dyck path is $D^u$ and whose labels on the down-steps from left to right are derived from the indices of the $\downstep$'s in $\DyckWord$ from left to right.
\begin{lemma}
	$\DyckWord$ equals $\emph{\ldword}(D^{\ell})$.
\end{lemma}
\begin{proof}
	Note that $\chi(\DyckWord)$ equals $\dword(D^{\ell})$.
	Furthermore, by construction, the indices of the $\downstep$'s read from left to right in $\ldword(D^{\ell})$ coincide with the indices of the $\downstep$'s read from left to right in $\DyckWord$.
	Therefore, all we need to show is that the index of a $\upstep$ in $\DyckWord$ matches that of the corresponding $\upstep$ in $\ldword(D^{\ell})$.

	Let $W_{old}$ be the value of $\DyckWord$ at an intermediate stage during the execution of the algorithm just prior to a pop operation.
	Let $W_{new}$ be the value of $\DyckWord$ immediately after a single pop.
	Then $W_{new}={\upstep}_i \cdot W_{old}$ where $i$ is the minimum label amongst all nodes in $Q$ before the pop was performed.
	We discuss an alternate way to describe this value $i$.
	Let $A$ (respectively $B$) be the set of indices of $\downstep$s (respectively $\upstep$s) in $W_{old}$.
	Then $A$ consists of labels of the nodes that have been pushed into $Q$ (up until the intermediate stage we are considering), whereas $B$ consists of labels of nodes that have been both pushed and popped.
	Then $A\setminus B$ contains all labels corresponding to nodes in $Q$ before the pop.
	Thus, we infer that $i$ is the minimum element in $A\setminus B$.

	Now consider the sub-path of $D^{\ell}$ that corresponds to $W_{old}$.
	The labels on the down-steps are precisely the elements of $A$.
	The sub-path of $D^{\ell}$ corresponding to $W_{new}$ begins with an up-step. The preceding procedure where we computed the labels on up-steps  coincides precisely with our description of the procedure that endows up-steps with labels during the computation of the labeled Dyck word.
	This finishes the proof.
\end{proof}

In other words, Algorithm~\ref{alg:ltreetoldyck} outputs the labeled Dyck word of  $D^{\ell}$ given a labeled binary tree $T$.
We note certain features of this algorithm that are immediate from its description.
\begin{lemma}\label{lem:two features ltreetoldyck}
	Algorithm~\ref{alg:ltreetoldyck} possesses the following three features.
	\begin{enumerate}
		\item If the label of $\rootof{T}$ is $i$, then the last letter in $\DyckWord$ is ${\downstep}_i$.
		\item If $i$ is the label of the left child of the  node labeled $j$ in $T$, then $\DyckWord$ has a contiguous subword of the type ${\downstep}_{i}{\downstep}_{j}$.
		\item If $i$ is the label of the right child of the  node labeled $j$ in $T$, then $\DyckWord$ has a contiguous subword of the type ${\downstep}_{i}{\upstep}_{j}$.
	\end{enumerate}
\end{lemma}

The reader is invited to compare the statement of Lemma~\ref{lem:two features ltreetoldyck} to that of Lemma~\ref{lem:two features ldycktoltree}.
This comparison naturally leads us to our next theorem.
From this point on, we identify $\ltreetoldyck(T)$ with the labeled Dyck path $D^{\ell}$ instead of $\ldword(D^{\ell})$.
Our next result establishes that $\ltreetoldyck$ is a bijection and its inverse is the map $\ldycktoltree$.

\begin{theorem}
	$\emph{\ltreetoldyck}: \ltree_n \to \ldyck_n$ is a bijection and its inverse is given by $\emph{\ldycktoltree}$.
\end{theorem}
\begin{proof}
	As $\ldyck_n$ and $\ltree_n$ have the same cardinality, to establish the claim it suffices to show that  $\ldycktoltree\circ \ltreetoldyck$ is the identity map.
	Let $D^{\ell}$ be the labeled Dyck path corresponding to the labeled Dyck word  output by Algorithm~\ref{alg:ltreetoldyck} for some $T\in \ltree_n$.
	From Lemmas~\ref{lem:two features ldycktoltree} and \ref{lem:two features ltreetoldyck}, it follows that $\ldycktoltree(D^{\ell})=T$.
	This finishes the proof.
\end{proof}
Figure~\ref{fig:master figure} should aid the reader in understanding our next theorem.
\begin{theorem}\label{thm:stat-preserving bijection}
	$ \emph{\ldycktoltree}\circ \emph{\prcttoldyck}: \SPRCT((2^n)) \to \ltree_n$ is a bijection that maps the ordered quadruple of statistics $(|\des_N|,|\des_S|,|\des_{NE}|,|\des_{SE}|)$ to $(\lasce,\ldes,\rasc,\rdes)$.
\end{theorem}
\begin{proof}
	It is clear that $\ldycktoltree\circ \prcttoldyck$ is a bijection.
	We focus on proving the second half of the theorem.
	Consider $\tau\in \SPRCT((2^n))$.
	Let $D$ be $\prcttoldyck(\tau)$ and let $T$ be $\ldycktoltree(D)$.
	Suppose first that  $i$ and $i+1$ both belong to the first column  in the $p$-th and  $q$-th row   in $\tau$ respectively.
	Then the $i$-th step of $D$ is a down-step labeled $p$ and the $i+1$-th step is down-step labeled $q$.
	In terms of $\ldword(D)$, this corresponds to an instance of a contiguous subword of the type ${\downstep}_p{\downstep}_q$.
	Lemma~\ref{lem:two features ldycktoltree} implies that $p$ is the label of the left child of the node labeled $q$ in $T$.
	If $i\in \des_N(\tau)$, then $p>q$ and
	the edge in $T$ joining the nodes labeled $p$ and $q$ is a left ascent.
	Thus we infer that {$|\des_N(\tau)|=\lasce(T)$}. A similar argument proves that {$|\des_S(\tau)|=\ldes(T)$}.

	Now suppose that  $i$ belongs to the first column in the $p$-th row while $i+1$ belongs to the second column in the $q$-th row  in $\tau$.
	Then the $i$-th step of $D$ is a down-step labeled $p$ and the $i+1$-th step is an up-step labeled $q$.
	In terms of $\ldword(D)$, this corresponds to an instance of a contiguous subword  ${\downstep}_p{\upstep}_q$.  Lemma~\ref{lem:two features ldycktoltree} implies that the node labeled $p$ is the right child of the node labeled $q$ in $T$.
	If $i\in \des_{NE}(\tau)$, then $p>q$ and the edge joining the nodes labeled $p$ and $q$ is a right ascent.
	Thus, we infer that {$|\des_{NE}(\tau)|=\rasc(T)$}. A similar argument proves that {$|\des_{SE}(\tau)|=\rdes(T)$}.
\end{proof}
\begin{figure}[htbp]
	\includegraphics[scale=0.9]{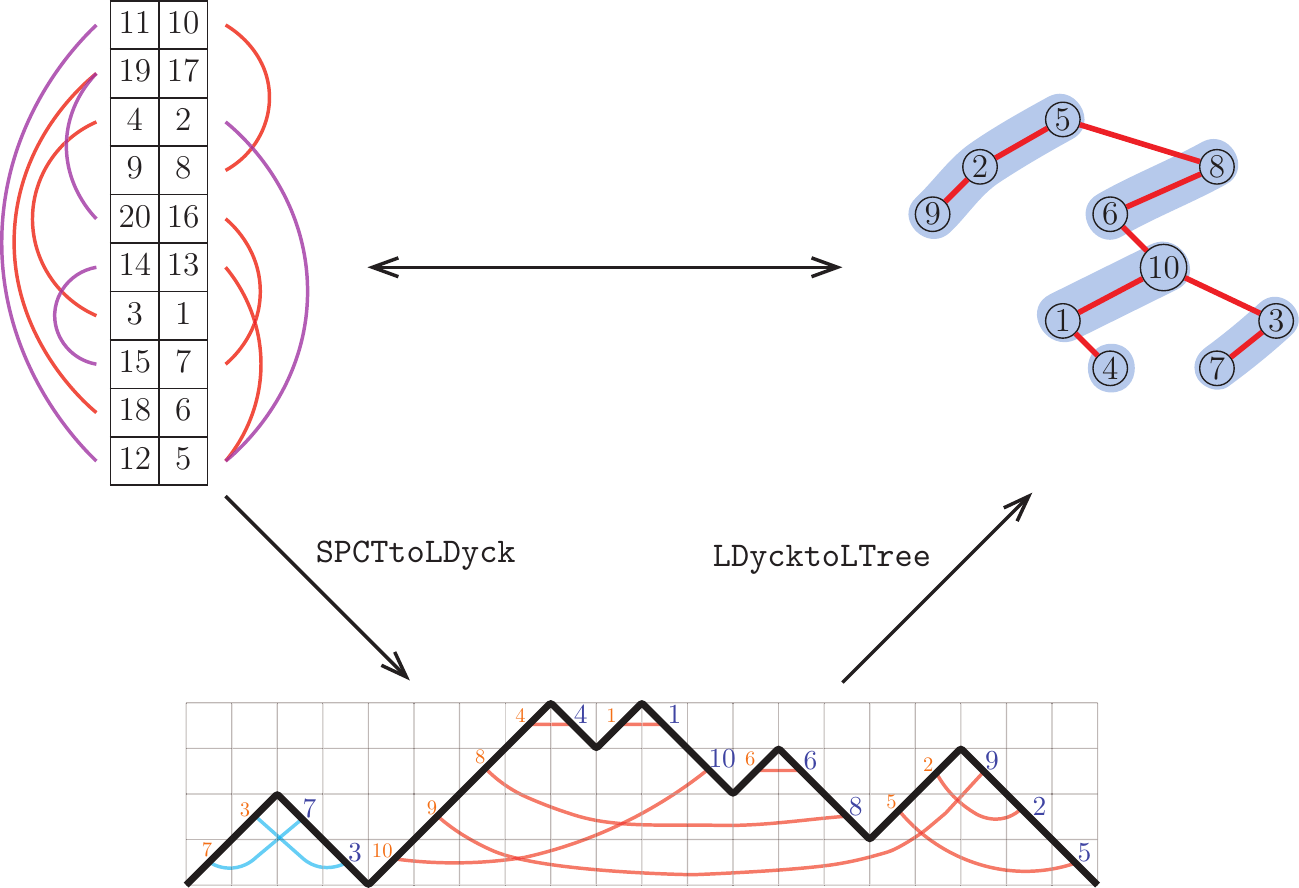}
	\caption{The statistic-preserving bijection between $\SPRCT((2^n))$ and $\ltree_n$ by way of labeled Dyck paths.}
	\label{fig:master figure}
\end{figure}

We return to our question at the beginning of this section, that of enumerating sink tableaux of shape $(2^n)$.
Recall that a $2$-columned tableau $\tau$ is a sink if {$|\des_{NE}(\tau)|=0$}.
By Theorem~\ref{thm:stat-preserving bijection}, the set of sink tableaux of shape $(2^n)$ is in bijection with the set of labeled binary trees on $n$ nodes that have no right ascents.
By a result originally due to  Pak-Postnikov \cite{PakPostnikov} (see also \cite{Pak-slides}), the number of such trees is $(n+1)^{n-1}$.
Thus, we arrive at the following result.
\begin{theorem}\label{thm:standardized 2-column}
	The set of ordered pairs $(\st_1(\tau),\st_2(\tau))$ as $\tau$ ranges over all  $2$-columned {$\SPRCT$s} of size $2n$ is in bijection with the set of labeled binary trees on $n$ nodes with no right ascents. Thus, its cardinality is $(n+1)^{n-1}$.
\end{theorem}
In the next section, we show that the  set of ordered pairs in Theorem~\ref{thm:standardized 2-column}  does in fact coincide with the set of compatible pairs.
To this end, we establish that for every $(\st_1(\tau),\st_2(\tau))$ obtained from some $\tau \in \SPRCT((2^n))$, we can construct an $\SRCT$ $\tau'$ that has two adjacent columns, say $i$ and $i+1,$ satisfying $\st_i(\tau')=\st_1(\tau)$ and $\st_{i+1}(\tau')=\st_2(\tau)$.
An alternative characterization for  $(\st_1(\tau),\st_2(\tau))$ using pattern-avoidance is key.

\section{Allowable pairs of permutations and acyclic graphs}\label{sec:allowable acyclic}
We begin with a lemma characterizing standardized first and second column words of $2$-columned tableaux  in terms of pattern-avoidance.
The {straightforward} proof relies on the triple condition and is omitted.
\begin{lemma}\label{lem:tau pattern avoidance}
	Consider $\tau \in \SPRCT ((2^n))$. Denote the entry in the $i$-th row and $j$-th column by $\tau_{(i,j)}$. Then the following properties hold.
	\begin{enumerate}
		\item For $1\leq i<j\leq n$, If
		$\tau_{(i,1)} >\tau_{(j,1)}$, then $\tau_{(i,2)} > \tau_{(j,2)}$.
		\item  For $1\leq i<j<k \leq n$, if
		$\tau_{(i,1)} <\tau_{(j,1)}< \tau_{(k,1)}$, then we cannot have that $\tau_{(i,2)}>\tau_{(k,2)}>\tau_{(j,2)}$.
	\end{enumerate}
\end{lemma}

Consider $\sigma,\gamma\in \sgrp{n}$.
We say that $(\sigma,\gamma)$ is \emph{$(21,12)$-avoiding} if $\sigma(i)>\sigma(j)$ for integers $1\leq i<j\leq n$ implies that $\gamma(i)>\gamma(j)$.
Similarly, we say that it is \emph{$(123,312)$-avoiding} if $\stan(\sigma(i)\sigma(j)\sigma(k))=123$ for integers $1\leq i<j<k\leq n$ implies that $\stan(\gamma(i)\gamma(j)\gamma(k))\neq 312$.
The statement that $(\sigma,\gamma)$ is $(21,12)$-avoiding is equivalent to saying that $\sigma \leq_{L} \gamma$ as it implies that $\Inv(\sigma)\subseteq \Inv(\gamma)$.
We call an ordered pair of permutations $(\sigma,\gamma)$ an \emph{allowable pair} provided it is both $(21,12)$ and $(123,312)$-avoiding.
Lemma~\ref{lem:tau pattern avoidance} states that $(\st_1(\tau),\st_2(\tau))$ is an allowable pair for any $2$-columned tableau $\tau$.
\begin{remark}
	The choice of terminology may be explained by noticing that a pair $(\sigma,\gamma)$ is allowable in our sense if and only if the pair $(\sigma^c,\gamma^c)$ is allowable in the sense of \cite{AtkinsonThiyagarajah}. Here $\sigma^{c}$ refers to the \emph{complement} of $\sigma$, that is, the permutation obtained by replacing $i$ with $n-i+1$ if $\sigma\in \sgrp{n}$.
	Under the complementation map, our pattern-avoidance characterization for allowable pairs  is equivalent to the one noted in the remark preceding \cite[Lemma 3]{ALW}.
 \end{remark}

 Our next aim is to show that the set of allowable pairs is equal to the set of ordered pairs of permutations considered in Theorem~\ref{thm:standardized 2-column}.
 We achieve this by associating a directed acyclic graph  to an allowable pair.
 The acyclicity of this graph implies that its node set can be endowed with the structure of a poset as follows.
 Given nodes $v,w$, we define $v\geq w$ if there exists a directed path from $v$ to $w$.
 For the graph we construct, a linear extension of the associated poset, known as a \emph{topological sort} in theoretical computer science, is a $2$-columned tableau.
 We provide the details next.

Given permutations {$\sigma_1,\ldots ,\sigma_k\in \mathfrak{S}_n$} where $k\geq 2$, we say that the sequence $(\sigma_1,\ldots,\sigma_k)$ is an \emph{allowable sequence} of \emph{length} $k$ if  $(\sigma_j,\sigma_{j+1})$ is allowable for all $1\leq j\leq k-1$.
With an allowable sequence {$(\sigma_1,\ldots,\sigma_k)$} we associate a directed simple graph $\mathbf{G}\coloneqq \mathbf{G}(\sigma_1,\dots,\sigma_k)$ as follows.
Consider the composition diagram of shape $(k^n)$.
Place a node in the middle of each cell of this diagram.
These nodes together form the node (or vertex) set of $\mathbf{G}$.
We reference each node by the pair $(i,j)$ where $i$ is the row and $j$ is the column to which this node belongs.
The edge set of $\mathbf{G}$ is determined by $(\sigma_1,\dots,\sigma_k)$ as follows.
\begin{enumerate}
\item For  all $1\leq i\leq n$ and $1\leq j\leq k-1$, draw a directed edge from $(i,j)$ to $(i,j+1)$.
We refer to these edges as \emph{horizontal edges}.
\item For all $1\leq j\leq k$, draw a directed edge from
$(p,j)$ to $(i,j)$ where $1\leq i\neq p\leq n$  are such that $\sigma_{j}(i)<\sigma_{j}(p)$.
 We refer to these edges as \emph{vertical edges}.
\item For all $2\leq j\leq k$, draw a directed edge from
$(p,j)$ to $(i,j-1)$ where $1\leq i< p\leq n$  are such that $\sigma_{j}(i)<\sigma_{j}(p)$.
We refer to these edges as \emph{diagonal edges}.
\end{enumerate}
For any $1\leq j\leq k$ and $1\leq i\leq n$, the nodes $(i,j)$ are said to belong to the $j$-th column.
Note that all horizontal edges are oriented from west to east, whereas all diagonal edges are oriented from southeast to northwest.
Furthermore, they connect nodes that belong to consecutive columns.

Figure~\ref{fig:acyclic} depicts the graph $\mathbf{G}$ associated to the allowable sequence $(123,132,321)$.
Note that it does not contain any directed cycles.
Thus, its nodes can be assigned distinct numbers from $1$ through $9$ such that all arrows point from nodes with larger labels to those with smaller labels.
While this labeling is not necessarily unique, the  conditions on the edge set guarantee that we obtain an $\SPRCT$ of shape $(3^3)$ in this manner.
We emphasize here that the condition for the existence of a diagonal edge is such that the triple condition, {using the version stated in \cite[Definition 4.1]{HLMvW-1}}, holds.
 \begin{figure}[htbp]
	\includegraphics[scale=0.6]{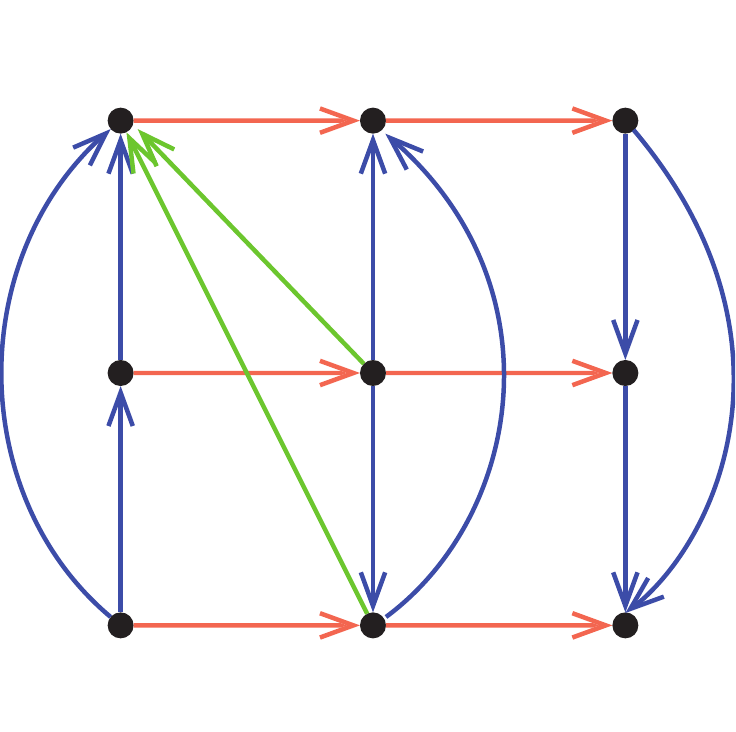}
	\caption{The directed graph $\mathbf{G}((123,132,321)).$}
	\label{fig:acyclic}
\end{figure}

\begin{remark}
Henceforth, whenever we use the term cycle, we mean a simple directed cycle.
Thus we do not allow repeating edges or nodes (except the initial and terminal nodes that are the same). Also, an edge directed from node $A$ towards node $B$ is denoted by $\overrightarrow{AB}$.
\end{remark}

\begin{theorem}\label{thm:G 2 permutation acyclic}
Given an allowable pair $(\sigma_1,\sigma_2)$, we have that $\mathbf{G}(\sigma_1,\sigma_2)$ is acyclic.
\end{theorem}
\begin{proof}
Throughout this proof, let $\mathbf{G}\coloneqq \mathbf{G}(\sigma_1,\sigma_2)$.
Additionally, let the subgraph of $\mathbf{G}$ induced by the vertical edges be $\mathbf{G}'$.
Note that $\mathbf{G}'$ is acyclic.
Thus,  any cycle in $\mathbf{G}$ must use at least one diagonal edge.
Furthermore,  if there are edges $\overrightarrow{AB}$ and $\overrightarrow{BC}$ in $\mathbf{G}'$, then there is a directed edge $\overrightarrow{AC}$.
This implies that two consecutive vertical edges in a cycle in $\mathbf{G}$ may be replaced by a single vertical edge.
Applying this reduction repeatedly to any cycle results in a (potentially) new cycle   where no two consecutive edges are vertical.
We call such a cycle \emph{reduced}.
Note that $\mathbf{G}$ is acyclic if and only if it does not contain a reduced cycle.

We first show by contradiction that $\mathbf{G}$ does not contain a reduced 3-cycle.
Assume that there is a reduced 3-cycle in $\mathbf{G}$.
This cycle must involve exactly one diagonal edge.
Suppose this diagonal edge is directed from $A\coloneqq (i_1,2)$ to $B\coloneqq (i_2,1)$ where $i_1>i_2$.
Let $C$ and $D$ be the nodes $(i_1,1)$ and $(i_2,2)$ respectively.
Then we have the configuration shown in Figure~\ref{fig:3cyc} in $\mathbf{G}$.
\begin{figure}[ht]
\includegraphics[scale=0.5]{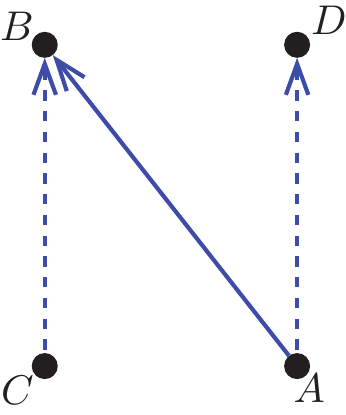}
\caption{Configuration in a potential 3-cycle.}
\label{fig:3cyc}
\end{figure}
As the edge $\overrightarrow{AB}$ is diagonal, we conclude there is a vertical edge $\overrightarrow{AD}$.
As the pair $(\sigma_1,\sigma_2)$ is $(21,12)$-avoiding, there exists a vertical edge $\overrightarrow{CB}$.
A reduced 3-cycle that involves $\overrightarrow{AB}$ must have the third node participating in the cycle being either $C$ or $D$.
As the horizontal edges are oriented eastwards, $\mathbf{G}$ {cannot contain} a reduced 3-cycle.

Next, we show that $\mathbf{G}$ does not contain a reduced $4$-cycle.
Assume that $\mathbf{G}$ contains a reduced $4$-cycle.
Such a cycle must involve exactly one diagonal edge, one horizontal edge and two non-consecutive vertical edges.
 Suppose it involves nodes $A$, $B$, $C$ and $D$ where the edge $\overrightarrow{AB}$ is diagonal.
 Let $A\coloneqq (i_1,2)$ and $B\coloneqq (i_2,1)$ where $i_1>i_2$.
Figures~\ref{fig:case1}, ~\ref{fig:case2}, and ~\ref{fig:case3} depict all the possibilities for such a reduced $4$-cycle.
We analyze each case separately.

\begin{figure}[ht]
\begin{minipage}[b]{0.30\textwidth}
\begin{centering}
\includegraphics[scale=0.45]{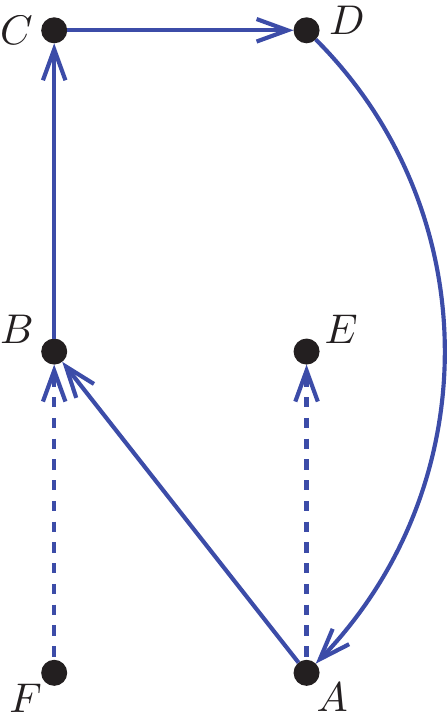}
\caption{Case 1}
\label{fig:case1}
\end{centering}
\end{minipage}
\begin{minipage}[b]{0.30\textwidth}
\begin{centering}
\includegraphics[scale=0.45]{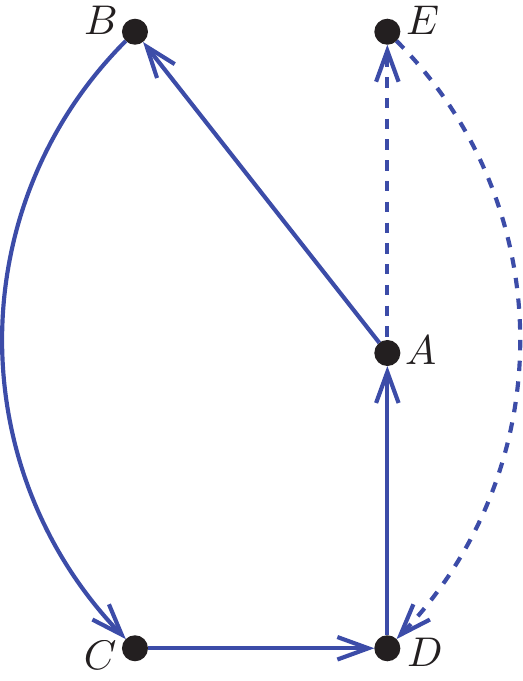}
\caption{Case 2}
\label{fig:case2}
\end{centering}
\end{minipage}
\begin{minipage}[b]{0.30\textwidth}
\begin{centering}
\includegraphics[scale=0.45]{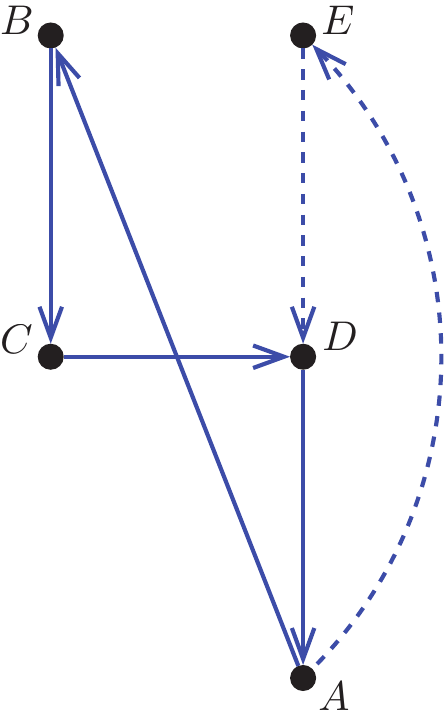}
\caption{Case 3}
\label{fig:case3}
\end{centering}
\end{minipage}
\end{figure}

In Case~1 shown in Figure~\ref{fig:case1}, let $E$ and $F$ be the nodes $(i_2,2)$ and $(i_1,1)$ respectively.
Let $C$ and $D$ be the nodes $(i_3,1)$ and $(i_3,2)$ respectively where $i_3<i_2$.
As $\overrightarrow{AB}$ is diagonal, following the argument used in establishing the non-existence of a reduced 3-cycle, we infer  the existence of vertical edges $\overrightarrow{AE}$ and $\overrightarrow{FB}$ respectively.
Observe that the vertical edges $\overrightarrow{FB}$ and $\overrightarrow{BC}$ together imply that $\sigma_1(i_3)<\sigma_1(i_2)<\sigma_1(i_1)$.
Similarly, the vertical edges $\overrightarrow{DA}$ and  $\overrightarrow{AE}$ together imply that $\sigma_2(i_2)<\sigma_2(i_1)<\sigma_2(i_3)$.
These two sequences of inequalities contradict the fact that $(\sigma_1,\sigma_2)$ is $(123,312)$-avoiding.
Therefore, the 4-cycle shown in Figure~\ref{fig:case1} cannot occur in $\mathbf{G}$.

 We will treat Cases 2 and 3 in Figures~\ref{fig:case2} and \ref{fig:case3} uniformly.
 Let $C$ be the node $(i_3,1)$ where $i_3> i_2$.
 Note that $i_3\neq i_1$ as that would give rise to a 3-cycle, which is impossible.
 Let $E$ be the node $(i_2,2)$.
 As $(\sigma_1,\sigma_2)$ is $(21,12)$-avoiding, the edge $\overrightarrow{BC}$ implies the existence of the vertical edge $\overrightarrow{ED}$.
  The diagonal  edge $\overrightarrow{AB}$ implies  the existence of the vertical edge $\overrightarrow{AE}$.
 Now note that the nodes $A,D,E$ are involved in a $3$-cycle in $\mathbf{G}'$, an impossibility.

At this point we have established that $\mathbf{G}$ does not contain any reduced 3-cycle or 4-cycle.
Hence it does not contain a $3$-cycle or a $4$-cycle.
We  proceed by induction to show that $\mathbf{G}$ does not contain a reduced $k$-cycle for $k\geq 5$.
Consider a reduced $k$-cycle in $\mathbf{G}$ whose nodes read by following the directed edges clockwise are  {$A_1, A_2, \ldots, A_k$}  where we further assume that the edge $\overrightarrow{A_1A_2}$ is diagonal.
There exist two possibilities for the edge $\overrightarrow{A_2A_3}$.
If it is horizontal, then the non-existence of a 3-cycle in $\mathbf{G}$ implies that there is a vertical edge $\overrightarrow{A_1A_3}$.
This in turn yields a shorter cycle of length $k-1$ on the nodes {$A_1,A_3,\ldots, A_k$}.
If, on the other hand, the edge from $\overrightarrow{A_2A_3}$ is vertical, then we are guaranteed that the edge $\overrightarrow{A_3A_4}$ is horizontal.
As $\mathbf{G}$ contains no reduced $4$-cycles, we infer the existence of the vertical edge $\overrightarrow{A_1A_4}$.
Thus, from the original cycle we obtain a shorter cycle of length $k-2$ on the nodes {$A_1,A_4,\ldots, A_k$}.

It follows that if $\mathbf{G}$ contains a reduced $k$-cycle, then it contains a strictly shorter cycle on a subset of the original set of nodes participating in the cycle.
As $\mathbf{G}$ contains no 3-cycles  or 4-cycles, it does not contain any reduced $k$-cycles for $k\geq 5$.
This finishes the proof.\end{proof}
As $\mathbf{G}(\sigma_1,\sigma_2)$ is acyclic for an allowable pair $(\sigma_1,\sigma_2)$ where $\sigma_1,\sigma_2\in \sgrp{n}$, we infer that there exists $\tau\in \SPRCT((2^n))$ such that $(\st_1(\tau),\st_2(\tau))=(\sigma_1,\sigma_2)$.
Using Theorem~\ref{thm:standardized 2-column}, we conclude that the number of allowable pairs $(\sigma_1,\sigma_2)$ with $\sigma_1,\sigma_2\in \sgrp{n}$ is $(n+1)^{n-1}$.
Of course, this fact was proved before \cite{ALW, AtkinsonThiyagarajah, Hamel} but our proof is different from the aforementioned ones.

Armed with the above insight, we will  show that allowable pairs are the same as compatible pairs.
Prior to that, we need the following generalization of Theorem~\ref{thm:G 2 permutation acyclic}.

\begin{theorem}\label{thm:general G acyclic}
Consider the allowable sequence {$(\sigma_1,\ldots, \sigma_k)$} for $k\geq 2$.
The directed graph {$\mathbf{G}(\sigma_1,\ldots, \sigma_k)$} is acyclic.
\end{theorem}
\begin{proof}
Throughout this proof, let {$\mathbf{G}\coloneqq \mathbf{G}(\sigma_1,\ldots, \sigma_k)$}.
Let $\mathbf{G}_h$ (respectively $\mathbf{G}_v$) be the subgraph of $\mathbf{G}$ induced by the horizontal edges (respectively vertical edges).
All edges in $\mathbf{G}_h$ are oriented eastwards, whereas all edges in $\mathbf{G}_v$ go from larger values to smaller values in terms of the permutations in our allowable sequence.
Thus, $\mathbf{G}_h$ and $\mathbf{G}_v$ are both acyclic.

We will establish the assertion in the theorem by induction on $k$.
The base case $k=2$ is Theorem~\ref{thm:G 2 permutation acyclic}.
Assume that the claim holds for all allowable sequences of length $i$ where $2\leq i<k$.
Suppose for the sake of contradiction that $\mathbf{G}$ is not acyclic.

Let $\mathbf{C}$ be a cycle in $\mathbf{G}$.
Observe that $\mathbf{C}$ must contain at least one node each from the first column and the $k$-th column.
If it did not, then $\mathbf{C}$ would in fact be a cycle in either {$\mathbf{G}(\sigma_1,\ldots,\sigma_{k-1})$} or {$\mathbf{G}(\sigma_2,\ldots,\sigma_k)$}, considered as subgraphs of $\mathbf{G}$ in the natural manner.
Our inductive hypothesis forbids either possibility.

Since diagonal edges connect nodes in consecutive columns,  $\mathbf{C}$ has nodes in all columns from  $1$ through $ k$.
Consider the induced subgraph of $\mathbf{C}$ formed from nodes that belong to the first and second columns.
This induced subgraph is a disjoint union of directed paths, say {$\mathbf{P}_1,\ldots, \mathbf{P}_m$}.
Suppose that, amongst the aforementioned paths, the paths {$\mathbf{P}_{i_1},\ldots, \mathbf{P}_{i_r}$} are those that contain at least one node from the first column.
As our cycle is simple, the initial and terminal nodes of each such path are distinct and furthermore, belong to the second column.
Thus, each $\mathbf{P}_{i_j}$ for $1\leq j\leq r$ contains at least one diagonal edge and one horizontal edge.
As $\mathbf{G}(\sigma_1,\sigma_2)$ considered as a subgraph of $\mathbf{G}$ is acyclic, the vertical edge that connects the initial and terminal node of each $\mathbf{P}_{i_j}$ for $1\leq j\leq r$ acquires a unique orientation.
Replacing each $\mathbf{P}_{i_j}$ with the vertical edge connecting its initial and terminal nodes gives another cycle $\mathbf{C'}$ in $\mathbf{G}$.
Our construction implies that  $\mathbf{C'}$ does not involve any node in the first column.
Thus, {$\mathbf{G}(\sigma_2,\ldots,\sigma_k)$} considered as a subgraph of $\mathbf{G}$ contains $\mathbf{C'}$, which is a contradiction.
This finishes the proof.
\end{proof}

The acyclicity of {$\mathbf{G}(\sigma_1,\ldots,\sigma_k)$}  implies that we can label its nodes with distinct positive integers from $1$ through $kn$ such that all edges in $\mathbf{G}$ point from nodes with larger labels to those with smaller ones.
Figure~\ref{fig:tableau given allowable} demonstrates the resulting $\SPRCT$  for the directed graph from Figure~\ref{fig:acyclic}.
\begin{figure}[htbp]
\centering
\includegraphics[scale=0.8]{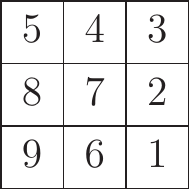}
\caption{An $\SPRCT$ corresponding to $\mathbf{G}((123,132,321))$.}
\label{fig:tableau given allowable}
\end{figure}

\begin{theorem}\label{thm:existence of prct}
Given an allowable sequence {$(\sigma_1,\ldots,\sigma_k)$} where {$\sigma_1,\ldots,{\sigma _k} \in \sgrp{n}$}, there exists an $\SPRCT$  {$\tau$ of} shape $(k^n)$ such that $\st_i(\tau)=\sigma_i$ for $1\leq i\leq k$.
\end{theorem}
To obtain a similar statement for $\SRCT$s, we require that $\sigma_1$ be the identity permutation in $\sgrp{n}$.
First, we use the weak Bruhat order to extend allowable pairs to allowable sequences whose first coordinate is the identity.
\begin{lemma}\label{lem:extending allowable pairs}
Given an allowable pair $(\delta_1,\delta_2)$ where $\delta_1,\delta_2\in \sgrp{n}$, there exists an allowable sequence {$(\sigma_1,\ldots,\sigma_k)$} where {$k\geq 3$} such that $\sigma_{k-1}=\delta_1$, $\sigma_{k}=\delta_2$ and $\sigma_1=\epsilon_n$, { and $\delta_1=\epsilon_n$ if $k=2$}.
\end{lemma}
\begin{proof}
Consider a maximal chain $\sigma_1\prec_L \sigma_2 \prec_L \cdots \prec_L \sigma_{k-1}=\delta_1$ in the weak Bruhat order, where $\sigma_1=\epsilon_n$.
We claim that {$(\sigma_1,\ldots,\sigma_{k-1})$} is an allowable sequence.
It suffices to show that if $\gamma_1,\gamma_2\in \sgrp{n}$ satisfy $\gamma_1\prec_L \gamma_2$, then $(\gamma_1,\gamma_2)$ is an allowable pair.

Since $\gamma_1\prec_L\gamma_2$, we know that $\gamma_2=s_p\gamma_1$ for some $p$ where $p$ comes before $p+1$ in $\gamma_1$.
As inversions in $\gamma_1$ are also inversions in $\gamma_2$,
we infer that $(\gamma_1,\gamma_2)$ is $(21,12)$-avoiding.
We proceed to show that $(\gamma_1,\gamma_2)$ is $(123,312)$-avoiding as well.

Suppose there are $1\leq i<j<k\leq n$ such that $\gamma_1(i)<\gamma_1(j)<\gamma_1(k)$.
We claim that we cannot have $\gamma_2(j)<\gamma_2(k)<\gamma_2(i)$.
Suppose to the contrary.
As $\gamma_2$ is obtained from $\gamma_1$ by swapping $p$ and $p+1$, we have that $\gamma_1$ and $\gamma_2$ differ in exactly two positions.
If neither $p$ nor $p+1$ belong to $\{\gamma_1(i),\gamma_1(j),\gamma_1(k)\}$, then $\gamma_2(i)<\gamma_2(j)<\gamma_2(k)$ as well.
If both $p$ and $p+1$ belong to $\{\gamma_1(i),\gamma_1(j),\gamma_1(k)\}$, then $\stan(\gamma_2(i)\gamma_2(j)\gamma_2(k))$ is a permutation that is distance $1$ from the identity in the Hasse diagram of the weak Bruhat order on $\sgrp{3}$.
In particular, $\stan(\gamma_2(i)\gamma_2(j)\gamma_2(k))\neq 312$.
If exactly one of $p$ or $p+1$ belongs to $\{\gamma_1(i),\gamma_1(j),\gamma_1(k)\}$, then  $\stan(\gamma_1(i)\gamma_1(j)\gamma_1(k))=\stan(\gamma_2(i)\gamma_2(j)\gamma_2(k))$.
Thus, $(\gamma_1,\gamma_2)$ is $(123,312)$-avoiding and we conclude that $(\gamma_1,\gamma_2)$ is an allowable pair.

The preceding argument implies that {$(\sigma_1,\ldots,\sigma_{k-1}=\delta_1)$} is an allowable sequence and therefore, {$(\sigma_1,\ldots,\sigma_{k-1},\delta_2)$} is an allowable sequence as well.
\end{proof}

We have the following corollary of Lemma~\ref{lem:extending allowable pairs} which establishes that allowable pairs and compatible pairs are the same notion.
\begin{corollary}\label{cor:existence of srcts}
Given an allowable pair $(\sigma,\gamma)$ where $(\sigma,\gamma)\in \sgrp{n}$, there exists an $\SRCT$ $\tau$ of shape $(k^n)$  where $k\geq 2$ and such that {$\st_{k-1}(\tau)=\sigma$ and $\st_k(\tau)=\gamma$}.
\end{corollary}
\begin{proof}
Construct an allowable sequence {$(\sigma_1,\ldots,\sigma_{k-1},\sigma_k)$} where $\sigma_1=\epsilon_n$, $\sigma_{k-1}=\sigma$ and $\sigma_k=\gamma$.
Then {$\mathbf{G}(\sigma_1,\ldots,\sigma_k)$} is acyclic.
The conclusion now follows.
\end{proof}

\section{Final remarks}
\begin{enumerate}
\item As mentioned earlier, the representation theory of $H_n(0)$ in type $A$ is related to the algebra of quasisymmetric functions.
Indeed, there is a map $ch$ called the \emph{quasisymmetric characteristic}  \cite{DKLT} that associates a quasisymmetric function to an $H_n(0)$-module.
Mimicking the proof of \cite[{Theorem 5.2}]{TvW}, one can establish that
$$ch(\mathbf{S}_{\alpha})=\sum_{\tau\in \SPRCT(\alpha)} F_{\comp(\tau)}.$$
 Here $\comp(\tau)$ is the composition of $n$ naturally associated with the subset $\des(\tau)$ of $[n-1]$ and $F_{\gamma}$ is the fundamental quasisymmetric function indexed by  a composition $\gamma$ \cite[{Chapter 7}]{stanley-ec2}.
 It would be interesting to investigate the algebraic/combinatorial properties of $ch(\mathbf{S}_{\alpha})$ in addition to the representation-theoretic properties of the $\mathbf{S}_{\alpha}$.

 \item In the introduction we mentioned that the study of ascents and descents on labeled binary trees is {tied to} the enumeration of chambers in various Coxeter deformations.
 The number of tableaux in $\SPRCT((2^n))$ is $n!\Cat_{n}$, which equals the number of regions in the Catalan arrangement in $\mathbb{R}^n$ defined by the hyperplanes $x_i-x_j=0,\pm 1$ for $1\leq i<j\leq n$.
 The sink tableaux of shape $(2^n)$ are distinguished representatives of the equivalence classes under $\sim_{(2^n)}$ and by Theorem~\ref{thm:standardized 2-column}, there are $(n+1)^{n-1}$ many of them.
 This is the number of regions in the Shi arrangement in $\mathbb{R}^n$ defined by the hyperplanes $x_i-x_j=0, 1$ for $1\leq i<j\leq n$.
 In view of this, it would be interesting to understand how the coarsening of regions of the Catalan arrangement into the regions of the Shi arrangement relates to grouping of tableaux in $\SPRCT((2^n))$ into equivalence classes under $\sim_{(2^n)}$.
 Our $0$-Hecke operators translate into operators on regions of the Catalan arrangement and this viewpoint merits further study.
\end{enumerate}

\section*{Acknowledgements}
We would like to thank Sara Billey, Ira Gessel and Sean Griffin for extremely helpful discussions. We would also like to thank the anonymous referees for their valuable feedback and suggestions.

\def\cprime{$'$}

\end{document}